\documentclass[a4paper,10pt]{article}
\usepackage[english]{babel}
\usepackage[utf8]{inputenc}
\usepackage{amsmath}
\usepackage{indentfirst}
\usepackage{fancyhdr}
\usepackage{stmaryrd}
\usepackage{amssymb}
\usepackage{amsthm}
\usepackage[dvips]{graphicx}
\usepackage{color}
\usepackage{cancel}
\usepackage[dvipsnames]{xcolor}
\usepackage{eucal}
\usepackage{latexsym}
\usepackage{chngcntr}
\usepackage{faktor}
\usepackage{microtype}
\usepackage{newlfont}
\usepackage{enumitem}
\usepackage{hyperref,geometry,mathtools} 
\usepackage{todonotes}

\usepackage{amsmath, amsfonts, amsthm,amssymb,  bbm, dsfont, lmodern}

\usepackage{amsmath}
\usepackage{amsthm}
\usepackage{wela}
\usepackage{tikz-cd}
\usepackage{multicol}
\usepackage{cancel}
\usepackage{mathtools}
\usepackage{graphicx}
\usepackage{MnSymbol}

\newcommand{\e}{\epsilon}
\newcommand{\im}{\mathrm{i}\,}

\newcommand\restr[2]{{
  \left.\kern-\nulldelimiterspace 
  #1 
  \vphantom{\big|} 
  \right|_{#2} 
  }}

\allowdisplaybreaks
\theoremstyle{plain}
\newtheorem{lem}{Lemma}
\newtheorem{teo}[lem]{Theorem}
\newtheorem{prop}[lem]{Proposition}

\theoremstyle{definition}

\newtheorem{rmk}[lem]{Remark}

\newcommand{\vet}[2]{\left[\footnotesize{\begin{matrix}#1 \\ #2 \end{matrix}}\right]}
\newcommand{\uno}{\mathrm{Id}}

\newcommand{\bR}{\mathbb{R}}

\newcommand{\bT}{\mathbb{T}}
\newcommand{\bZ}{\mathbb{Z}}
\newcommand{\bZd}{\mathbb{Z}^d}
\newcommand{\bN}{\mathbb{N}}

\newcommand{\bC}{\mathbb{C}}

\newcommand{\rdj}{\mathrm{\dj}}

\newcommand{\cL}{\mathcal{L}}

\newcommand{\de}{\mathrm{d}}
\newcommand{\pa}{\partial}

\newcommand{\cH}{\mathcal{H}}

\newcommand{\cF}{\mathcal{F}}

\numberwithin{equation}{section}
\counterwithin{lem}{section}

\title{\bf On the analyticity of the \\ 
Dirichlet-Neumann operator 
and  
Stokes waves
}
\begin{document}

 \author{Massimiliano Berti, Alberto Maspero, Paolo Ventura}

\date{}

\maketitle

\qquad \qquad \qquad \qquad \qquad \qquad \qquad 
\qquad
\qquad \qquad \qquad \quad {\it To the memory of Antonio Ambrosetti} 

\bigskip

\medskip 
\noindent
{\bf Abstract:} We  prove an analyticity result for the Dirichlet-Neumann operator 
under space periodic boundary conditions  in any dimension in an unbounded domain 
 with infinite depth. 
We derive an analytic  bifurcation result of  analytic 
Stokes waves --i.e. space periodic traveling solutions-- of 
the 
water waves equations in deep water. 
\\[1mm]
{\it MSC 2020:} 76B15, 35B32, 35J05.  

 \tableofcontents

\section{Introduction and main results}

The Dirichlet-Neumann operator plays an important role in fluid mechanics, for example in the Craig-Sulem-Zhakarov \cite{CS,Zak1} 
formulation of the water waves equations (cfr. Section \ref{sec:stokes}),    
and in several other branches of analysis, as in 
the theory of inverse problems. Roughly speaking it  is defined as the linear 
operator  which maps 
the Dirichlet datum of a harmonic function in a 
given domain into the normal derivative at its boundary 
 (Neumann datum). 
The Dirichlet-Neumann operator is nonlinear with respect to the boundary of the domain. 
In view of many applications it is important 
to determine 
its regularity in different function spaces. 

Several results about the analyticity of the Dirichlet-Neumann operator 
acting between Sobolev spaces,  with respect to the variation of the boundary, 
have been proved, starting with the pioneering works of Coifmann-Meyer \cite{CM}, Craig-Schanz-Sulem 
\cite{CSS}, Craig-Nicholls \cite{CN} and Lannes \cite{LannesLivre} where we
refer for an extended bibliography.  
We also mention the deep analysis of the Dirichlet-Neumann operator in \cite{AIM,AD,ABZ2,Wu1}, on which we will comment later.

\smallskip
The major 
aim of this paper is to prove a further analyticity result for the Dirichlet-Neumann operator 
$ G(\eta ) $ defined in \eqref{DN1}
on the unbounded domain $ \bT^d \times \{ y \leq \eta (x) \} $, where 
$ \bT^d := (\bR / 2 \pi \bZ)^d $ is the standard  $ d$-dimensional flat 
torus, in any space dimension $ d \geq 1  $. 
Assuming that $\eta (x) $ is analytic, we prove in 
Theorem \ref{thm:DN}  the analyticity of the 
map  $ \eta \mapsto G(\eta) $ acting between suitable spaces of analytic 
periodic functions. The delicate point of this result 
is that $ \eta $ and $ \psi $ are assumed to 
have the {\it same} regularity (if  
$ \eta $ is more regular than $ \psi $ the result is simpler). 
Following  Lannes \cite{Lannes,LannesLivre} and Alazard-Burq-Zuily \cite{ABZ}
we make use of a regularizing diffeomorphism 
to flatten the domain to the half cylinder,  in which the transformed 
harmonic function solves  a perturbed elliptic equation. Then 
the proof relies on a perturbative approach to invert the transformed Laplacian 
over suitable spaces of functions $ u(x,y) $ which are analytic in $ x $, with Sobolev 
regularity in $ y $ and 
decay to zero as $ y \to -\infty $, cfr. \eqref{spaziok}.  
The key step is obtain  linear elliptic regularity estimates for the 
Poisson equation in these spaces, see Lemma \ref{lem:ell5}. 
Then the elliptic estimates for the modified problem are obtained  
by a perturbative argument 
 differently from  \cite{ABZ}. 

\smallskip

As a consequence of  Theorem \ref{thm:DN}, 
we derive an analytic bifurcation result of analytic Stokes waves --i.e. space periodic traveling solutions, which look stationary in a  moving frame with constant speed-- of  the pure gravity  water waves equations in infinite depth, see Theorem \ref{LeviCivita}. 
Existence of traveling waves which are constant in one space dimension, i.e.  
are $ 1$-dimensional waves, 
dates back to  classical works of 
Levi-Civita \cite{LC}, Nekrasov \cite{Nek} and Struik \cite{Struik},
 in the twenties of the last century.  Then Lewy \cite{Lewy} proved
that a traveling wave  which is at least $ C^1 $
is actually analytic. Theorem \ref{LeviCivita} 
proves in addition that small amplitude  Stokes waves depend 
analytically on the amplitude taking values in a space of analytic functions. 
In finite depth and with  surface tension, a result of this kind 
is proved in Nicholls-Reitich \cite{NR} by a power series expansion approach. 

In this paper we deduce Theorem \ref{LeviCivita}  by the analytic 
Crandall-Rabinowitz 
bifurcation theorem from a simple eigenvalue, as presented in
the book of Ambrosetti-Prodi \cite{AP}, thanks to the analytic estimates 
of the Dirichlet-Neumann operator obtained in Theorem 
\ref{thm:DN}.

\smallskip

In addition to their interests per se --traveling waves have
 fundamental importance in fluid mechanics--, 
 these results have  been used in the study of the
Benjiamin-Feir instability of the  Stokes waves in \cite{BMV}.

\smallskip

We now state precisely our results. Along this paper we use the following notation.
We denote the spatial variables by 
$ (x,y)  \in \bT^d \times \bR $, $ d \geq 1 $,  where $ \bT^d := (\bR/ 2 \pi \bZ)^d $ 
is  the standard flat torus.
 The symbol $\nabla$ denotes the gradient
$$
\nabla := (\pa_{x_j})_{j=1, \ldots, d} \qquad \text{and} 
\qquad \Delta := \sum_{j=1}^d \pa_{x_j}^2 \, , 
\quad \Delta_{x,y} := \Delta + \pa_y^2 \, . 
$$
 A dot will denote the standard scalar product in $\bR^d$.
Moreover 
$ \bN := \{1, 2, \ldots \} $ and $ \bN_0 = \{0\} \cup \bN  $.

\subsection{Dirichlet-Neumann operator}

We consider the cylindrical domain 
\begin{equation}\label{Deta}
{\cal D}_\eta := \big\{  (x,y) \in \bT^d \times \bR \ : \ y < \eta (x) \big\} \, , \quad d \geq 1  \, , 
\end{equation}
delimited by the graph $ \partial {\cal D}_\eta =  \{ y = \eta (x) \} $ of a periodic function $ \eta (x) $, and, given  
a periodic Dirichlet datum   
$ \psi (x) $,  we consider the unique harmonic function $\Phi(x,y)$ 
solving the system
\begin{equation}\label{Phi}
\begin{cases}
\Delta_{x,y} \Phi = 0  & \mbox{ in } {	\cal D}_\eta \\  
\Phi(x, y ) = \psi(x) & \mbox{ at } y = \eta(x) \\
\pa_y \Phi(x,y) \to 0 & \mbox{ as } y \to - \infty \, .
\end{cases}
\end{equation}
The Dirichlet-Neumann operator $G(\eta) $ is then defined as the linear operator
\begin{equation}
\label{DN1}
\begin{aligned}
[G(\eta) \psi](x) & := \sqrt{1+|\nabla \eta|^2} \,  \pa_n \Phi_{|y = \eta (x) } \\
& = (\pa_y \Phi) (x, \eta(x)) - \nabla \eta(x) \cdot (\nabla \Phi)(x, \eta(x))
\end{aligned}
\end{equation}
where $ n $ denotes the exterior normal 
$$ 
n :=  \frac{1}{\sqrt{1 + |\nabla \eta |^2}} \vet{- \nabla \eta}{1}  \, , 
\quad 
\pa_n := \frac{1}{\sqrt{1 + |\nabla \eta |^2}}(\pa_y - \nabla \eta \cdot \nabla ) \, . 
$$ 
The reason of the name 
``Dirichlet-Neumann"  is that the operator $G(\eta) $  maps the Dirichlet datum $ \psi (x) $ of 
the harmonic function  $ \Phi (x,y) $  into the (normalized) normal derivative  
$ \partial_n \Phi $ 
 at the 
boundary $ \partial {\cal D}_\eta = \{ y = \eta (x) \} $ (Neumann datum). 

\begin{rmk}
In \eqref{Phi}  it is equivalent to require the boundary condition 
$ \nabla \Phi (x,y) \to 0 $
as $ y \to - \infty $, see Remark \ref{remdecay}. Actually 
$ \nabla \Phi (x,y)  $ decays to zero exponentially fast as $ y \to - \infty $. 
\end{rmk}

Simple algebraic properties of the Dirichlet-Neumann operator are recalled in Appendix \ref{app:DN}.

Since Calderon it is known that 
the Dirichlet-Neumann operator $ G(\eta) $ is, if $ \eta $ is a $ C^\infty $ function, 
a classical pseudo-differential operator, elliptic of order $ 1 $, with an asymptotic expansion
in classical decreasing symbols. 
For the flat surface $ \eta (x) = 0 $, the Dirichlet-Neumann operator is the Fourier multiplier 
$$
G(0) = |D| = (-\Delta)^{\frac12} 
$$
as follows by the elementary calculus \eqref{ell3}. 
In space dimension $ d = 1 $ 
the Dirichlet-Neumann operator is equal to $|D|$ up to infinitely many 
times regularizing operators, see e.g.  \cite{BM,BBHM}.
If $ \eta (x) $ has a finite smoothness, Lannes \cite{Lannes,LannesLivre} 
proved  an analogous expansion in symbols  with finite smothness.  

The Dirichlet-Neumann operator is a nonlinear map with respect to the boundary 
of the domain
$ \partial {\cal D}_\eta $.  
The analytic dependence with respect to $ \eta $
of the Dirichlet-Neumann operator $ \eta \mapsto G(\eta )$ 
has been first established in the two dimensional setting by Coifman-Meyer 
\cite{CM}, and in the three dimensional setting by
Craig, Schanz and Sulem \cite{CSS}, showing 
that, if  $ \eta \in C^{k+1} $, $ \psi \in H^{k+1}$, $ k \in \bN $, then
$ G(\eta)[\psi] \in H^k $ is analytic in  $C^{k+1} \cap \{ \| \eta \|_{C^1}< r \} $ for $r$  sufficiently small.

In view of of application to water waves 
Craig-Nicholls \cite{CN}, Wu 
\cite{Wu1,Wu2},  and Lannes \cite{Lannes,LannesLivre} proved, 
 with different approaches, that if 
 $ \eta, \psi $ have the same Sobolev regularity $H^s$
then $  G(\eta) [\psi ] \in H^{s-1} $.  
In particular Lannes proved tame estimates using {\it regularizing} diffeomorphisms
to  straighten the domain. 

The paralinearization of $ G(\eta) \psi $, which enables to prove optimal
estimates for the action of the Dirichlet-Neumann operator, has been obtained in 
Alazard-Metivier \cite{AIM}, 
Alazard-Delort \cite{AD},  
and Alazard-Burq-Zuily \cite{ABZ1,ABZ2} 
in rough domains, using a variational analysis to construct the solution 
and applying elliptic regularity theory. 
The paralinearization of  the Dirichlet-Neumann operator in $ d = 1 $
with a multilinear expansion in $ \eta $ is proved in Berti-Delort \cite{BD}, by using a
paradifferential parametrix \`a la Boutet de Monvel.   

Finally we mention the work of Alazard-Burq-Zuily \cite{ABZ}
for the study of the Dirichlet-Neumann operator acting in  analytic function spaces,
making use of a regularizing diffeomorphism as in \cite{Lannes}, variational methods and
 elliptic regularity analysis.

\smallskip

In this paper we prove an analyticity result  (Theorem \ref{thm:DN})
for the Dirichlet-Neumann map $ \eta \mapsto  G(\eta) \psi $ in 
the cylindrical domain $ {\cal D}_\eta $ defined in \eqref{Deta}, 
acting between spaces of periodic analytic functions defined in \eqref{def:Hsigmas} below.
We suppose that the functions  $\eta$ and $\psi$ 
belong to the spaces of   periodic 
functions
\begin{equation} \label{def:Hsigmas}
H^{\sigma, s} := H^{\sigma, s}(\bT^d):= 
\Big\{ u (x) = \sum_{k \in \bZd}  u_k e^{\im k \cdot x} \, : \ \  
\| u \|_{H^{\sigma, s}}^2 := \sum_{k \in \bZd}  e^{2\sigma |k|_1}  \langle k \rangle^{2s}
\, |  u_k|^2 < \infty  \Big\}
\end{equation}
where,  for any $ k = (k_1, \ldots, k_d) \in \bZ^d $, we set 
$$
|k|_1 := |k_1|+\dots+|k_d| \, , \quad
\langle k \rangle := \max(1, |k|) \, , \quad 
| k| := \Big( \sum_{j=1}^d k_j^2 \Big)^{1/2} \, .
$$
Clearly, if the dimension $ d = 1 $ then $ |k| = |k |_1 $. 

If $ \sigma = 0 $ the space $ H^{0, s}  $ is the usual Sobolev space $ H^s $. 
If  $ \sigma > 0 $, a periodic function $ u (x) $ belongs to  
$H^{\sigma, s}(\bT^d)$, if and only if 
it admits an analytic extension in the strip $ |y|_{\infty} 
:= \max\{ |y_1|, \ldots, |y_d| \} < \sigma $ and 
the traces at the boundaries $ u( \cdot  + \im y ) $,  $ |y|_\infty = \sigma $,  
belong to   the Sobolev space $H^s := H^s(\bT^d)$.  
In  Appendix 
\ref{App1A} we prove this characterization, together with the property that 
 the spaces $ H^{\sigma,s} $ 
form, for $ s > d/ 2 $, an algebra with respect to the product of functions
and satisfy tame estimates.
\smallskip

The main result of this section is the following theorem. 

Let $ B^{\sigma, s}(r) $ denote the open ball in
$ H^{\sigma,s} $ of center $ 0 $ and radius $ r > 0  $. 
\begin{teo}\label{thm:DN}
{\bf (Dirichlet-Neumann operator)}
Let $\sigma\geq 0$ and $s $, $ s_0$ 
such that $s+\frac12,\, s_0 \in \bN $,  
and  $s-\frac32\geq s_0 > \frac{d+1}{2} $. Then there exists $\e_0 :=\e_0(s)>0$ such that 
the Dirichlet-Neumann operator map\footnote{$H^{\sigma, s}\cap B^{\sigma, s_0}(\e_0) $ is an open set in the $H^{\sigma,s}$ topology. } 	
$$
\eta \mapsto G(\eta) \, , \quad 
H^{\sigma, s}\cap B^{\sigma, s_0+\frac32}(\e_0)  \to \cL(H^{\sigma, s}, H^{\sigma, s-1}) \, , 
$$ 
is analytic and fulfills the tame estimate 
\begin{equation}\label{tameGeta}
\| G(\eta)\psi \|_{H^{\sigma,s-1}} \leq C(s)\big( \|\psi\|_{H^{\sigma,s}} +\|\eta \|_{H^{\sigma,s}}\|\psi\|_{H^{\sigma,s_0+\frac32}}  \big) \, .   
\end{equation}
\end{teo}
We remark that, in Theorem \ref{thm:DN}, the functions 
$ \eta, \psi $ have the same analytic regularity. 
The proof of such result, given in Section \ref{sec:DN},   
relies on a regularizing flattening method (following \cite{Lannes, ABZ}) together with a perturbative argument in 
suitable functional spaces. 

\subsection{Stokes waves}\label{sec:stokes}

As an application of Theorem \ref{thm:DN} we prove that 
$ 1$-dimensional Stokes waves solutions of the pure gravity 
water waves equations in deep water 
are analytic functions belonging to the spaces $ H^{\sigma,s} $, 
and moreover depend  analytically with respect to 
 the amplitude  parameter.
  Clearly  $ 1 $-dimensional traveling waves are also $ 2$d-traveling waves 
which are constant in one space direction, so it extends to higher dimensional Stokes waves.
We first present the 
water waves equations. 
\\[1mm]{\bf The pure gravity water waves equations.}
We consider the Euler equations for a bi-dimensional incompressible,
inviscid, irrotational fluid under the action of  gravity,
filling the 
region ${ \cal D}_\eta $ defined in  \eqref{Deta} with $ d = 1 $,   
\begin{equation}\label{water}
\begin{cases}
\partial_t \Phi + \frac12 
\big( ( \pa_x \Phi )^2 + (\pa_y \Phi)^2 \big) + g \eta = 0  \ \ 
\qquad     \text{ at} \ y = \eta (x)  \cr
\partial_t \eta = \partial_y \Phi - (\partial_x \eta) \, 
(\partial_x \Phi) \qquad \qquad \qquad \quad \ \,  \text{at} \  y = \eta (x)  \cr
\Delta_{x,y} \Phi =0 \qquad  \qquad \qquad \qquad \      \qquad  \qquad  \ \quad \text{in} \ {\cal D}_{\eta}  \cr
\pa_y \Phi \to 0  \qquad  \qquad \ \ \,  \qquad \qquad \qquad \qquad \quad \   \text{ as} \  
y \to  - \infty  \, , 
\end{cases}
\end{equation}
where $ g >  0 $ is the acceleration of gravity. 
The irrotational velocity field is the gradient  
of the harmonic scalar potential $\Phi=\Phi(t,x,y) $,  
determined by its trace $ \psi(t,x)=\Phi(t,x,\eta(t,x)) $ at the free surface
$ y = \eta (t, x ) $.
Actually $\Phi (t, \cdot ) $ is 
the unique solution of the elliptic equation \eqref{Phi}. 
The time evolution of the fluid is determined by the first 
two boundary conditions in \eqref{water}  at the free surface. 
The first states  that the pressure of the fluid  
is equal, at the free surface, to the constant atmospheric pressure  (dynamic boundary condition) and the second one  
that the fluid particles  remain, along the evolution, on the free surface   (kinematic
boundary condition). 

As shown by Zakharov \cite{Zak1} and Craig-Sulem \cite{CS}, 
the evolutionary system \eqref{water} amounts to 
the following equations for the unknowns $ (\eta (t,x), \psi (t,x)) $,  
\begin{equation}\label{WWeq}
 \eta_t  = G(\eta)\psi \, , \quad 
  \psi_t  =  
- g \eta - \dfrac{\psi_x^2}{2} + \dfrac{1}{2(1+\eta_x^2)} \big( G(\eta) \psi + \eta_x \psi_x \big)^2 \, , 
\end{equation}
where  $G(\eta)$ is 
the Dirichlet-Neumann operator in \eqref{DN1}.
In addition the equations \eqref{WWeq} are the Hamiltonian system
\begin{equation}\label{HSWW}
\pa_t \eta = \nabla_\psi \mathcal{H} \, , \quad \pa_t \psi = - \nabla_\eta \mathcal{H} \, , 
\end{equation}
 where $ \nabla_{\eta}, \nabla_{\psi} $ 
 denote the $ L^2$-gradients of the Hamiltonian
$$  
\mathcal{H}(\eta,\psi) := 
 \frac12 \int_{\mathbb{T}} \left( \psi \,G(\eta)\psi +g \eta^2 \right) \de x \, , 
$$
which  is the sum of the kinetic  energy (cfr. \eqref{DNSP})
 and potential gravitational energy  of the fluid. 
 Actually, as proved in \cite{CS,Zak1}, the $ L^2$-gradient with respect to 
 $ \eta  $ of the kinetic energy
\begin{equation} \label{kine}
K(\eta, \psi) := \frac12 (\psi, G(\eta) \psi)_{L^2}  \stackrel{\eqref{DNSP}} 
= \frac12 \int_{{\cal D}_\eta} 
|\nabla \Phi|^2 \, \de x \, , 
\end{equation} 
is equal to
\begin{equation}\label{Kform}
\nabla_\eta K(\eta, \psi) =  -  \frac12 \psi_x^2  +  \dfrac{1}{2(1+\eta_x^2)} \big( G(\eta) \psi + \eta_x \psi_x \big)^2 \, ,
\end{equation} 
yielding the equivalence between \eqref{HSWW} and \eqref{WWeq}. 
 
 We also remark that the water waves equations \eqref{WWeq} 
are  invariant under space translations namely, by 
\eqref{DN-inv}, 
$$ 
{\mathcal H} \circ \tau_\theta = {\mathcal H } \, , \quad \forall \theta \in \bR^d \, . 
$$   
In addition, the water waves equations
are reversible with respect to the involution 
$$
\rho\vet{\eta(x)}{\psi(x)} := \vet{\eta(-x)}{-\psi(-x)}, \quad \text{i.e. }
\mathcal{H} \circ \rho = \mathcal{H} \, , 
$$
as a consequence of \eqref{DN-rev}.
\\[1mm]
{\bf The Stokes waves.}
Noteworthy solutions of \eqref{WWeq} are the so-called  Stokes waves, namely traveling 
solutions of the form 
\begin{equation}\label{profiles}
\eta(t,x)=\breve \eta(x-ct) \, , \quad 
\psi(t,x)=\breve \psi(x-ct) \, , 
\end{equation}
for some real  $c $ (the speed)  and  $2\pi$-periodic functions  $(\breve \eta (x), \breve \psi (x)) $ (the profiles).
In a reference frame in translational motion with constant speed $c$, 
 the water waves equations \eqref{WWeq} then become,
 by using the translation invariance property 
\eqref{DN-inv},  
\begin{equation}\label{travelingWW}
\eta_t  = c\eta_x+G(\eta)\psi \, , \quad 
 \psi_t  = c\psi_x - g \eta - \dfrac{\psi_x^2}{2} + \dfrac{1}{2(1+\eta_x^2)} \big( G(\eta) \psi + \eta_x \psi_x \big)^2  \, .
\end{equation}
The Stokes waves profiles $(\breve \eta, \breve \psi)$ in \eqref{profiles}  
are then equilibrium steady solutions 
of \eqref{travelingWW}, namely solve the system
\begin{equation}\label{travelingWWstokes}
c\eta_x+G(\eta)\psi = 0 \, , \quad 
c\psi_x - g \eta - \dfrac{\psi_x^2}{2} + 
\dfrac{1}{2(1+\eta_x^2)} \big( G(\eta) \psi + \eta_x \psi_x \big)^2 
= 0 \, .
\end{equation}
The next theorem is the main bifurcation result of small amplitude Stokes waves 
proved in this paper.  
We denote by $B(r):= \{ x \in \bR \colon \  |x| < r\}$ the real ball with center 0 and radius $r$. 
 \begin{teo}\label{LeviCivita}
{\bf (Stokes waves)} For any  $ \sigma \geq 0 $, 
$ s > 5/ 2 $ and $k\in\bN$, there exists  $ \e_0 := \e_0 (\sigma,s,k) > 0 $
and a unique family  of 
 solutions 
 $$
 (\eta_\e(x), \psi_\e(x), c_\e) \in H^{\sigma,s} (\bT)\times H^{\sigma,s}(\bT)\times \bR 
 $$ 
  of the system  \eqref{travelingWWstokes}, 
  parameterized by  
  $|\e| \leq \e_0$, 
  such that
  \begin{enumerate}
  \item  
 the map 
 $ \e \mapsto (\eta_\e, \psi_\e, c_\e)  $, 
$  B(\e_0) \to H^{\sigma,s} (\bT)\times H^{\sigma,s}(\bT)\times \bR $ 
 is analytic; 
 \item 
 $\eta_\e (x) $ is even,  $ \eta_\e (x) $ has zero average, 
 $\psi_\e (x) $ is odd; 
\item the solutions $ (\eta_\e(x), \psi_\e(x), c_\e) $ have the expansion
\begin{equation}\label{HH}
 (\eta_\e(x), \psi_\e(x)) = \e (\sqrt{k} \cos (kx), \sqrt{g} \sin (kx)) + O(\e^2) \, , \quad c_\e \to 
 \sqrt{\frac{g}{k}} 
 \ \ \text{as} \ \ \e \to 0 \, .
 \end{equation}
 \end{enumerate}
\end{teo}

 Theorem \ref{LeviCivita} is proved in Section \ref{sec:Sto}.
 Let us make some comments on the result.
 \\[1mm]
1. As already mentioned in the introduction, the first rigorous bifurcation 
proof of  small amplitude  Stokes waves for pure 
  gravity water waves  goes back to  Levi-Civita \cite{LC} and  
 Nekrasov \cite{Nek} in deep water, and Struik \cite{Struik} in finite depth.
 We refer to the monographs of Ambrosetti-Prodi \cite{AP}
 and Buffoni-Toland \cite{BuT} for a complete presentation.
Concerning regularity, it is known  since  Lewy \cite{Lewy} that a Stokes wave which is at least $ C^1 $ 
is actually analytic.  Theorem \ref{LeviCivita} 
proves in addition that the Stokes waves $  (\eta_\e(x), \psi_\e(x)) $  depend 
analytically on the amplitude $ \e $ taking values in a space of analytic functions $ H^{\sigma,s} 	\times H^{\sigma,s}   $. 
In finite depth and in presence of surface tension, an analyticity result of this kind 
 is proved in Nicholls-Reitich \cite{NR}, by a power series expansion.
We also mention Plotnikov-Toland \cite{PlotTol} for  related results about analytic continuation of Stokes waves.

 Existence of traveling water waves has been also proved by 
 Zeidler \cite{Zei} under the effect of capillary forces and 
 Martin  \cite{Martin}, Walh\'en \cite{Wh1} also for constant vorticity flows. 
 We expect that, thanks to  Theorem \ref{thm:DN},  
 an analyticity result for the Stokes waves, analogous to  Theorem \ref{LeviCivita},  
holds also in these cases. 
\\[1mm]
2. Higher order Taylor expansions of the Stokes waves in $ \e $ are known, see e.g. 
\cite{Fenton}, \cite{NS}, \cite{NR}. 
 We remark that Theorem \ref{LeviCivita} proves the convergence of the Taylor 
 series of the Stokes waves in $ \e $, taking values in spaces of analytic periodic functions. 
\\[2mm]
3.  {\it Quasi-periodic traveling waves.}
More general $ 1$d  time quasi-periodic traveling Stokes waves 
  have been recently obtained in Berti-Franzoi-Maspero \cite{BFM,BFM2}, with or without
  surface tension,   and Feola-Giuliani \cite{FG}, by means 
  of a Nash-Moser implicit function iterative scheme.
We remark that these solutions  are not steady in any moving frame. This implies 
a small divisor problem.  
\\[1mm]
4. {\it Higher space dimension: existence.}
{For three dimensional fluids,   in addition to Stokes waves,  
also  traveling wave solutions which are 
nontrivially periodic  in both spatial directions  are known}, 
for example forming hexagonal patterns. 
 Their existence was first  proved in  Craig-Nicholls
\cite{CN} for  gravity-capillary  water waves, by applying 
variational bifurcation arguments \`a la Weinstein-Moser, exploiting the Hamiltonian
nature \eqref{HSWW} of the water waves equations.
The surface tension allows to apply, in the bifurcation analysis,  
the standard implicit function theorem. 
On the other hand the existence of 
$ 2$d pure gravity doubly-periodic traveling wave solutions is a small divisor problem. 
In this case, solutions with Sobolev regularity were constructed   by 
Iooss-Plotinkov
\cite{IP-Mem-2009,IP2} by means of a Nash-Moser implicit function theorem,
requiring suitable Diophantine conditions on the speed vector.
\\[1mm]
5. {\it 
 Regularity.}
In higher space dimensions a regularity result \`a la Lewy \cite{Lewy}, i.e. 
a traveling wave surface which is at least $ C^1 $
is actually analytic,  
has been  proved for gravity-capillary water waves
by Craig-Matei \cite{CM}. 
For 
pure gravity waves, a   result of this kind is false,  
 because the system is no more elliptic. This feature is the counterpart of the small divisor problem arising in the existence proof of Iooss-Plotinkov \cite{IP-Mem-2009,IP2}. 
Assuming Diophantine conditions on the speed vector, 
Alazard-Metivier  \cite{AIM} proved that the periodic traveling waves 
constructed in  \cite{IP-Mem-2009,IP2}, which have Sobolev regularity,  
are  indeed $ C^\infty $. 
\\[1mm]
6. We finally note that, for larger values of the amplitude 
$ \e $,  the regularity of the traveling wave
solutions 
may break down. Indeed it is well known that 
large traveling waves have cusps, as proved in the celebrated works about the Stokes 
conjecture of Amick, Fraenkel, Toland \cite{AFT} and Plotinkov \cite{Plot}.

\section{Analyticity of the Dirichlet-Neumann operator}\label{sec:DN}

In this section we prove Theorem \ref{thm:DN} concerning the analyticity of the Dirichlet-Neumann operator. 
The first step is to straighten  the free surface. 
\\[1mm]
{\bf Regularizing diffeomorphism.}
Following \cite{Lannes,ABZ} we apply the regularizing 
change of variables
\begin{equation}
\label{smooth.chanhe}
 x = x' \, , \qquad y = \rho(x', y') \, ,  \quad \rho(x', y') := y' + e^{y'|D|}\eta(x') \, , 
\end{equation}
where $e^{y|D|} $ is the Fourier multiplier
$$
\left(e^{y|D|} g\right)(x):=  \sum_{k \in \bZ^d}  g_k \, e^{y|k|} \, e^{\im k \cdot x} \, , 
\quad \forall \  g(x) = \sum_{k \in \bZ^d}  g_k  \, e^{\im k \cdot x} \, . 
$$
Note that 
$$
\rho(x', 0) = \eta(x') \, ,   \quad \lim_{y' \to - \infty} \rho(x',y')-y' = \eta_0 \, ,  
$$
and, since
$$ 
\pa_{y'} \rho(x',y') = 1 + e^{y|'D|} |D| \eta \, , 
$$
 if 
 $ \sup_{y'<0} \| e^{y'|D|} |D| \eta \|_{L^\infty (\bT^d)} < 1 $
the change of coordinates \eqref{smooth.chanhe}  is a diffeomorphism between  the domain $ {\cal D}_\eta = \{(x,y) \, \colon \,  y \leq \eta(x)\}$ and the flat half-cylinder  $\{ (x',y') 
\, \colon \,  y' \leq 0 \} = \bT^d \times \bR_{\leq 0}$ where $ \bR_{\leq 0} := (-\infty, 0] $. 
By the change of variables \eqref{smooth.chanhe}  the derivatives $\pa_y$ and $\nabla_x$ become respectively
$$
\Lambda_1 = \frac{1}{\pa_{y'} \rho} \pa_{y'} \, , \qquad
\Lambda_2 = \nabla_{x'} - \frac{\nabla_{x'} \rho}{\pa_{y'}\rho} \pa_{y'} \, ,
$$
and the transformed harmonic function 
$$
\varphi(x',y') := \Phi(x', y' + \rho(x',y')) 
$$
solves the elliptic problem 
\begin{equation}
\label{ell2.0}
\begin{cases}
(\Lambda_1^2 + \Lambda_2^2)\varphi = 0 \\
\varphi(x,0) = \psi(x)    \\
\pa_y \varphi(x,y) \to 0 & \mbox{ as } y \to - \infty \, .
\end{cases}
\end{equation}
By means of chain rule, system \eqref{ell2.0}  is rewritten (cfr.  \cite{ABZ})    
as the perturbed elliptic problem    (we rename the variables $x',y' $ as $ x,y$)
\begin{equation}
\label{ell2}
\begin{cases}
\Delta_{x,y} \varphi = F(\eta)[\varphi] 
  \\
\varphi(x,0) = \psi(x)    \\
\pa_y \varphi(x,y) \to 0 \quad \mbox{ as } y \to - \infty \, , 
\end{cases}
\end{equation}
where
\begin{equation}\label{def:f}
F(\eta)[\varphi]:=    \left( \alpha(\eta) \pa_y^2  + \beta(\eta) \Delta + \gamma(\eta) \cdot \nabla\pa_y + \delta(\eta ) \pa_y \right) \varphi
\end{equation}
with, since $ \nabla \rho(x,y) = e^{y|D|} \nabla \eta $ and 
$ \pa_{y} \rho(x,y) = 1 + e^{y|D|} |D| \eta $, 
 \begin{equation}
 \label{f}
 \begin{aligned}
&\alpha(\eta) :=  1- \frac{1+|\nabla \rho|^2}{\pa_y \rho} = 
\frac{e^{y|D|} |D| \eta- | e^{y|D|} \nabla \eta|^2}{1+e^{y|D|} |D| \eta } \, , 
\\
&\beta(\eta) := 1-\pa_y \rho = - e^{y|D|} |D| \eta \, , \\
&\gamma(\eta) := 2 \nabla \rho = 2 e^{y|D|}\nabla \eta \, , \\
&\delta(\eta) :=\frac{1}{\pa_y \rho } \Big( -2 \nabla\rho \cdot \nabla \pa_y \rho + \pa_y \rho \Delta \rho 
+ \frac{1+|\nabla \rho|^2}{\pa_y \rho} \pa_y^2 \rho \Big)\, . 
 \end{aligned}
 \end{equation}
In the new variables \eqref{smooth.chanhe},  the Dirichlet-Neumann operator defined in \eqref{DN1} becomes
\begin{equation}
\label{DN2}
\begin{aligned}
[G(\eta)\psi] (\cdot )& =
- \nabla\eta\cdot \nabla\varphi(\cdot, 0) + 
\frac{1+ |\nabla\eta|^2(\cdot)}{1+(|D|\eta)(\cdot)} \,  (\pa_y \varphi)(\cdot, 0) \, .
\end{aligned}
\end{equation}

\paragraph{Function spaces.}
In order to state our main existence result for the solutions  of \eqref{ell2},  
we introduce some function spaces. 
Given $s \in \bN_0 $, $ \sigma,  a  \geq 0 $, we define 
\begin{equation}
\label{spaziok}
\cH^{\sigma, s, a}:= 
\Big\{ 
u(x,y) = \sum_{k \in \bZ^d}   u_k (y) e^{\im k \cdot x}
\, \colon \, \bT^d \times (-\infty, 0] \to \bC \ \ \mbox{:} \ \    \| u \|_{\sigma, s, a} < \infty 
\Big\}  \,
\end{equation}
endowed with  the norm 
\begin{align}
\| u \|_{\sigma, s, a}^2 & := 
\sum_{j = 0}^s 
\| \pa_y^j u \|_{L^{2,a}( \bR_{\leq 0},H^{\sigma,s-j})}^2   \label{migliore} \\
& = 
\sum_{j = 0}^s 
\int_{-\infty}^0 \| \pa_y^j u(\cdot, y)\|_{H^{\sigma,s-j}}^2 \, e^{-2 a y} \de y  \notag \\ 
& = 
\sum_{j = 0}^s \int_{-\infty}^0 \sum_{k \in \bZd}
 \, e^{2 \sigma |k|_1 } \, \langle k \rangle^{2(s-j)} \, 
| \pa_y^j  u_k  (y)|^2 e^{2 a |y|} \de y  \notag \\  
& = 
\sum_{j = 0}^s \sum_{k \in \bZd}
e^{2 \sigma |k|_1 } \, \langle k \rangle^{2(s-j)} \,  \| \pa_y^j u_k \|_{L^{2,a}}^2   
\label{migliore4}
\end{align}
where,  given a Hilbert space $ X $, we have used the notation
\begin{equation}\label{defuL2a}
\| u \|_{L^{2,a}( \bR_{\leq 0},X)}^2  := \int_{-\infty}^{0} \| u(y) \|_X^2 e^{-2 a y} \,  \de y
= \int_{-\infty}^{0} \| u(y) \|_X^2 e^{2 a |y|} \, \de y  \, .
\end{equation}

\begin{rmk}
For $ \sigma = a = 0$,  the space $ H^{0,s,0} $ coincides with the 
Sobolev space $ H^s ( \bT^d \times \bR_{\leq 0 }) $ of 
$ L^2 $ functions $ u : \bT^d \times \bR_{\leq 0 } \to \bC $  
possessing weak derivatives $ \pa^\alpha u $ in $ L^2 $, for any multiindex $ \alpha \in \bN^{d+1}$ with modulus $ | \alpha| \leq s $, 
with equivalent norm
$ \| u \|_s^2 = \sum_{\alpha \in \bN^{d+1}, |\alpha| \leq s} \| \pa^\alpha u \|_{L^2}^2 $.
\end{rmk}

We point out that, for any $ s \in \bN $,  
$$
\|u\|_{\sigma,s,a}^2 = \|u\|^2_{L^{2,a}(\bR_{\leq 0}, H^{\sigma,s})} + \|\pa_y u\|_{\sigma,s-1,a}^2 \, ,
$$
and, by  \eqref{migliore} and $  \| \pa_{x_i} v \|_{H^{\sigma,s-1}} 
 \leq  \| v  \|_{H^{\sigma,s}}  $,   we directly get the following simple lemma. 
\begin{lem}\label{lem:prop}
Let $ s \in \bN $, $ \sigma \geq 0 $, $ a \geq 0 $. 
The linear maps
$$
\pa_{x_i} : \cH^{\sigma, s, a} \mapsto \cH^{\sigma, s-1, a} \, , \   \forall i =1,\dots,d, 
\quad \pa_y : \cH^{\sigma, s, a} \mapsto \cH^{\sigma, s-1, a} \, , 
$$
are continuous.
\end{lem}

We also denote 
\begin{equation}\label{coplusH}
\bC \oplus \cH^{\sigma, s, a}
:= \big\{ c + u(x,y) \, , \ c \in \bC \, , \ u \in \cH^{\sigma, s, a} \big\} \, ,\quad \Pi : \bC \oplus \cH^{\sigma, s, a} \to \cH^{\sigma, s, a}, \ \Pi [c + u] = u \, , 
\end{equation}
and, with a small abuse of notation, given a function 
$g  \in \bC \oplus \cH^{\sigma, s, a}$, we denote  its norm by
$\| g\|_{\sigma, s, a}:=  \| \Pi g \|_{\sigma, s, a} + |g - \Pi g | $. 
The function spaces $\cH^{\sigma, s, a}  $ and $\bC \oplus \cH^{\sigma, s, a} $ are  
modeled to mimic the decay 
of the harmonic function  
$ \underline{\varphi} $ in 
\eqref{ell3} as $ y \to - \infty $, cfr.  Lemma \ref{lem:und.varphi}. 

\smallskip

We now list a series of properties 
of the  spaces $\cH^{\sigma, s, a}$ used in the sequel; we defer their proofs in 
  Appendix \ref{sec:A2}.

\begin{lem}[{\bf Trace}]\label{lem:trace}
Let  $ \sigma \geq 0 $, $ s  \in \bR $. Then one has
\begin{equation}\label{uinfty}
\|u \|_{C^0(\bR_{\leq 0},H^{\sigma,s})}\leq \|u\|_{L^{2}(\bR_{\leq 0},H^{\sigma,s+\frac12})} + \|\pa_y u\|_{L^{2}(\bR_{\leq 0},H^{\sigma,s-\frac12})}.
\end{equation}
In particular, 
the  trace operator 
\begin{equation} \label{optraccia}
\Gamma (u):= u(\cdot, 0) := u\vert_{y= 0}
\end{equation}
is, for any  $s\in \bN_0$, $a \geq 0$,   a linear bounded 
 map between $ \cH^{\sigma, s+1, a} \to H^{\sigma, s+\frac12}$, 
satisfying
\begin{equation}
\label{est:trace}
\| \Gamma (u) \|_{H^{\sigma, s+ \frac12}} \leq   \| u \|_{\sigma, s+1, 0} \leq  \| u \|_{\sigma, s+1, a} \, . 
\end{equation}
\end{lem}

If  $s   > \frac{d+1}{2}$, the space $  \cH^{\sigma, s, a} $  is an algebra with respect to the product of functions and the following tame estimates hold. 
 
\begin{prop}[{\bf Tame}]\label{lem:algebray}
 Let $\sigma, a \geq 0$ and $s \geq s_0  > \frac{d+1}{2}$, $s,s_0 \in \bN$. 
 Then there exist positive constants $C_s \geq 1 $ (non-decreasing in $ s $)  such that, for any  
$u  \in \cH^{\sigma, s, 0}$ and  $ v \in \cH^{\sigma, s, a}$, 
\begin{equation}
\label{algebrassa}
\| u v \|_{\sigma, s, a} \leq C_s \big(
\| u \|_{\sigma, s, 0} \, \| v \|_{\sigma, s_0, a } +
  \| u \|_{\sigma, s_0, 0} \, \| v \|_{\sigma, s, a } \big) \, . 
\end{equation}
In particular one has
\begin{equation}\label{potenze}
\| u^j \|_{\sigma,s,a}\leq \big(2 C_s \|u\|_{\sigma,s_0,a}\big)^{j-1} \|u\|_{\sigma,s,a}\, , \quad \forall j \geq 1 \, .
\end{equation}
\end{prop}

The next lemma proves the continuity 
of the harmonic function  $e^{y|D|}g $, which solves the Dirichlet-Neumann 
elliptic problem  \eqref{ell3},
with respect to  the Dirichlet datum $ g $ at  $ y = 0 $.  

\begin{lem}[{\bf Harmonic propagator}]\label{lem:und.varphi}
Let $\sigma\geq 0$ and $s + \frac12 \in \bN$.
Then, for any $g \in H^{\sigma, s}$,  the function 
$$
( e^{y|D|}g)(x) := \sum_{k\in \bZ^d} g_k \, e^{y|k|} \, e^{\im k\cdot x}
$$  
belongs to 
$ \bC \oplus \cH^{\sigma, s+\frac12 , a} $, $a \in (0,1)$,    
and the linear map 
$$
H^{\sigma, s} \to  
\cH^{\sigma, s+\frac12, a} \, , 
\quad g \mapsto   \Pi [ e^{y|D|}g ] = e^{y|D|}g -  g_0 \, , 
$$ 
is continuous.
\end{lem}

We now come back to  Theorem \ref{thm:DN}. 
The  key result of its proof is the following proposition regarding the solution of the  elliptic problem \eqref{ell2}.

The parameter $a \in (0,1)$ 
plays a technical 
role  in studying the decay as $y \to -\infty$ of the solution of the elliptic problem 
\eqref{ell2.0} (see in particular Lemma \ref{lem:T}). 
{\em  In the sequel  we fix $a = \frac12$.}

\begin{prop}\label{prop:varphi}
Let $\sigma\geq 0$  and $s,\, s_0$ such that  
$s+\frac12 $, $ s_0 \in \bN $ and 
$s-\frac32\geq s_0 > \frac{d+1}{2} $. Then there exist $\e_0 :=\e_0(s)>0$ and, 
 for any $ \eta \in H^{\sigma,s}\cap B^{\sigma,s_0+\frac32}(\e_0)$ and $\psi \in H^{\sigma, s}$,
 a unique solution $\varphi \in \bC \oplus \cH^{\sigma, s+\frac12, a}$ 
  of the elliptic problem \eqref{ell2}, satisfying
\begin{equation}\label{tamevarphi}
\| \Pi \varphi \|_{\sigma,s+\frac12,a} \leq C(s)\big( \|\psi\|_{H^{\sigma,s}} +\|\eta \|_{H^{\sigma,s}}\|\psi\|_{H^{\sigma,s_0+\frac32}}  \big) \, . 
\end{equation}
Moreover  $\varphi = \Psi(\eta)[\psi]$, where 
$\Psi$ is an analytic map  $H^{\sigma,s}\cap B^{\sigma,s_0+\frac32}(\e_0) \to$ $
  \cL(H^{\sigma,s}, \, \bC \oplus \cH^{\sigma, s+\frac12, a} )$, and 
  $\Psi(0) \psi= e^{y|D|}\psi$.
\end{prop}
Postponing the proof of this proposition, we first use it to deduce Theorem  \ref{thm:DN}.

\smallskip

\noindent{\bf Proof of Theorem  \ref{thm:DN}.}
By Proposition \ref{prop:varphi}, for any  
$ \eta \in H^{\sigma, s}\cap B^{\sigma,s_0+\frac32}(\e_0)$ and $ \psi \in  H^{\sigma, s}$, 
there exists a unique solution $ \varphi \in \bC \oplus \cH^{\sigma, s+ \frac12, a} $ 
of 
\eqref{ell2}. The Dirichlet-Neumann operator is computed in \eqref{DN2}.
Since $ \varphi (x,0) = \psi (x) $, using the  trace operator $  \Gamma(u)  = u(\cdot, 0) $
in \eqref{optraccia}, and recalling the definition of $ \Pi $ in \eqref{coplusH}, 
we rewrite   \eqref{DN2} as 
\begin{align}
G(\eta)\psi   &= - \nabla\eta\cdot \nabla \psi+ 
 \frac{1+ |\nabla\eta|^2}{1+(|D|\eta)} \,  \Gamma[ \pa_y \varphi ] \notag 
 \\ \label{DN3bis}
 &=\underbrace{ - \nabla\eta\cdot \nabla \psi}_{=: G_1(\eta)\psi}
 +
 \underbrace{ \Gamma[ \pa_y \Pi \varphi ]}_{=: G_2(\eta)\psi} 
 + 
\underbrace{ \frac{|\nabla\eta|^2-(|D|\eta)}{1+(|D|\eta)} \,  \Gamma[ \pa_y \Pi \varphi ]}_{=:  G_3(\eta)\psi} \, .
\end{align}
We prove that each map 
\begin{equation}\label{Giana}
G_i : 
H^{\sigma, s}\cap B^{\sigma, s_0+\frac32}(\e_0)  \to \cL(H^{\sigma, s}, H^{\sigma, s-1}) \, ,
\ \  i =1, 2, 3 \, , \quad \text{is analytic} \, , 
\end{equation}
and fulfills the tame estimate
\eqref{tameGeta}.
Regarding $G_1(\eta)\psi$, it suffices to note that it is  linear in $\eta$ and  by \eqref{tameHs}, 
\begin{equation} \label{stima1} 
\|\nabla\eta\cdot \nabla \psi\|_{H^{\sigma,s-1}} \lesssim_s \|\eta \|_{H^{\sigma,s_0+\frac12}} \|  \psi\|_{H^{\sigma,s}} + \|\eta \|_{H^{\sigma,s}} \|  \psi\|_{H^{\sigma,s_0+\frac12}} \, . 
 \end{equation} 
Next we consider $G_2(\eta)\psi =  \Gamma[\pa_y \Pi \Psi(\eta)\psi] $. 
By Lemma \ref{lem:prop} and  \ref{lem:trace}, the map $\varphi \mapsto  \Gamma[\pa_y \Pi \varphi] \in \cL(\bC \oplus \cH^{\sigma, s+\frac12, a}, H^{\sigma, s-1})$ which, together with the 
analyticity  of $\eta \mapsto \Psi(\eta)$ stated in Proposition \ref{prop:varphi}, implies  the analyticity of $\eta \mapsto G_2(\eta)$ as in \eqref{Giana}.
Moreover by 
\eqref{est:trace}, Lemma \ref{lem:prop} and \eqref{tamevarphi}, we have
\begin{equation}\label{stima2} 
\| \Gamma[\pa_y  \Pi \varphi]\|_{H^{\sigma,s-1}} \leq \|\pa_y \Pi \varphi \|_{\sigma,s-\frac12,0} \lesssim_s 
 \|\psi\|_{H^{\sigma,s}} +\|\eta \|_{H^{\sigma,s}}\|\psi\|_{H^{\sigma,s_0+\frac32}}  \, .
\end{equation}
Finally consider $G_3(\eta)\psi = f(\eta) G_2(\eta)\psi$, where 
$f(\eta)$ is the multiplication operator by the function 
\begin{equation}\label{domsvif}
 f(\eta) = \frac{|\nabla\eta|^2-(|D|\eta)}{1+(|D|\eta)}
 =  \big( |\nabla\eta|^2-(|D|\eta) \big) \sum_{j=0}^\infty (-|D|\eta)^j \, .
 \end{equation}
By Lemma \ref{lem:surfacingsubalgebra} we have that 
$ \| (|D| \eta)^j \|_{H^{\sigma, s-1}}
\leq ( C(s) \| \eta \|_{H^{\sigma, s_0 + \frac32}})^j
   \|\eta \|_{H^{\sigma,s}} $ for any  $  j \in \bN $,  
and therefore  $ f(\eta)$ in \eqref{domsvif}
is bounded,  on the domain  $H^{\sigma, s}\cap B^{\sigma, s_0+\frac32}(\e_0)  $, by
\begin{equation}\label{boufeta}
\| f(\eta) \|_{H^{\sigma, s-1}}
\lesssim_s 
   \| \eta \|_{H^{\sigma,s}} \ .
\end{equation}
Moreover  $f(\eta)$ in \eqref{domsvif}
  is a series of analytic functions 
 uniformly convergent on the sets $B^{\sigma,s}(R) \cap B^{\sigma, s_0+\frac32}(\e_0) $, $\forall R >0$.
Thus, by Weierstrass theorem, 
$\eta \mapsto f(\eta)$ is analytic on
$B^{\sigma,s}(R) \cap B^{\sigma, s_0+\frac32}(\e_0)$, and, by the arbitrariness of $R$, on the whole open set 
$H^{\sigma, s}\cap B^{\sigma, s_0+\frac32}(\e_0)  \to H^{\sigma, s-1} $.
We conclude that also $G_3  (\eta) $  is analytic
as stated  in \eqref{Giana}. 
Finally,  \eqref{boufeta} and \eqref{stima2} imply that $ G_3 (\eta) $ satisfies
the tame estimate \eqref{tameGeta}.
\qed

\begin{rmk}
It  follows from the proof that $G(0)\psi = G_2(0)\psi = \Gamma[\pa_y \Pi \Psi(0)\psi]$, which, together with  $\Psi(0) \psi= e^{y|D|}\psi$,  recovers the identity $G(0)\psi = |D| \psi$.
\end{rmk}
\vspace{1em}

The final paragraph is devoted to the  proof of Proposition \ref{prop:varphi}.

\paragraph{Proof of Proposition \ref{prop:varphi}: the perturbative argument.}  
We look for a solution $\varphi$ of \eqref{ell2} of the form
\begin{equation}
\label{eq:varphi}
 \varphi(x, y) =  \underline{\varphi}(x, y) + u(x, y)
\end{equation}
where $ \underline{\varphi}$ is the harmonic solution of 
\begin{equation}
\label{ell3}
\begin{cases}
\Delta_{x,y} \underline{\varphi} =0  \\
\underline{\varphi}(x,0) = \psi(x) &  \\
\pa_y \underline{\varphi}(x,y) \to 0 & \mbox{ as } y \to - \infty  \, , 
\end{cases} 
\qquad
\mbox{ i.e. }  \quad  
\underline{\varphi}(x,y) := e^{y|D|}\psi(x) \, ,  
\end{equation}
whereas
$u $ solves the elliptic problem
\begin{equation}\label{ell4}
\begin{cases}
\Delta_{x,y} u = F(\eta)[\phi + u]  \, , \qquad   \\
u(x,0) =0 &  \\
\pa_y u(x,y) \to 0 \qquad \mbox{ as } y \to - \infty \, ,
\end{cases}
\end{equation}
with $\phi :=\underline{\varphi}$. 
The harmonic function $ \underline{\varphi} = e^{y|D|} \psi  $ 
ia estimated by Lemma \ref{lem:und.varphi}. 
\begin{rmk}\label{remdecay}
Also the derivative $ \pa_x  \underline{\varphi} (x,y) \to 0 $ as $ y \to - \infty $. 
Actually any solution of \eqref{Phi} satisfies $ \nabla \Phi (x,y) \to 0 $ as
$ y \to - \infty $. Indeed let $ a $ such that $ \bT \times \{ y = - a \} \subset {\cal D}_\eta $. 
Since the harmonic function 
$ \Phi (x, y )$ is analytic then $ \vartheta (x) := \Phi (x, -a)$   is analytic as well. Thus
(cfr. \eqref{ell3}) we can represent 
$ \Phi (x,y) = \sum_{k \in \bZd}   \vartheta_k \, e^{|k| (y+a)} \, e^{\im k\cdot x} $, which proves that
$$ 
\nabla \Phi (x,y) = \sum_{k \in \bZd \setminus \{0\}}  
\im k \,  \vartheta_k \, e^{|k| (y+a)} \, e^{\im k\cdot x}  \to 0 \quad \text{as}
\quad y \to - \infty \, , 
$$ 
actually  exponentially fast. 
\end{rmk}
The solution of system \eqref{ell4} is given  by  the following lemma. 
\begin{lem}\label{lem:lasol}
Let $\sigma\geq 0$ and $s,\, s_0$ such that  
$s+\frac12,\, s_0 \in \bN $ and $s-\frac32\geq s_0 > \frac{d+1}{2} $. Then there exist $\e_0 :=\e_0(s)>0$ and a unique analytic map
\begin{align*}
\eta \mapsto U(\eta) \, , \quad
U: H^{\sigma,s}\cap B^{\sigma, s_0+\frac32}(\e_0) \longrightarrow \cL \big( \bC \oplus \cH^{\sigma, s+\frac12,a}\big)\, ,
\end{align*}
such that $u = U(\eta)[\phi]=U(\eta)[\Pi \phi]$, with $\Pi $ in \eqref{coplusH},
solves  \eqref{ell4}, satisfying  
\begin{equation}\label{tameu}
\|\Pi U(\eta)[\phi]\|_{\sigma,s+\frac12,a} \leq C(s )\big( \|\eta \|_{H^{\sigma,s_0+\frac32}}\|\Pi \phi\|_{\sigma,s+\frac12,a} +\|\eta \|_{H^{\sigma,s}}\|\Pi \phi\|_{\sigma,s_0+2,a}  \big) \, . 
\end{equation}
\end{lem}
The proof of  Lemma \ref{lem:lasol} relies on 
Lemmata \ref{lem:ell5} and \ref{lem:non} below. 

Given a function $ g(x, y) $ defined in $ \bT^d \times (- \infty, 0 )$, we  first 
 consider the linear elliptic problem
\begin{equation}\label{ell5}
\begin{cases}
\Delta_{x,y} u = g  \\
u(x,0) =0 &  \\
\pa_y u(x,y) \to 0 & \mbox{ as } y \to - \infty  \, . 
\end{cases}
\end{equation}
The following key lemma is proved in Appendix \ref{sec:sol.ell}.
\begin{lem}[{\bf Elliptic regularity}]\label{lem:ell5}
Fix $\sigma \geq 0$, $s \in \bN_0 $  and $a \in (0,1)$. 
For any  $g \in \cH^{ \sigma, s,  a}$, the  elliptic problem 
\eqref{ell5}
has a unique solution 
$u := L (g)  \in \bC \oplus  \cH^{\sigma, s+2, a} $. 
The linear map 
$$
L  :  \cH^{ \sigma, s,  a}\to \bC \oplus  \cH^{ \sigma, s+2,  a} \, , \quad 
 g \mapsto L(g) \, , 
$$ 
is continuous, i.e. there exists $ C_{a} > 0 $ such that 
$ \| Lg \|_{ \sigma, s+2, a} \leq C_{a} \| g\|_{\sigma, s,a} $.
\end{lem}

Thanks to Lemma \ref{lem:ell5}, we recast the nonlinear elliptic problem \eqref{ell4} into the equation
\begin{equation}
\label{ell10}
\big(\uno - L \circ F(\eta)\big) [u] = L\circ F(\eta)[\phi] \, .
\end{equation} 
Note that the linear operator $\uno - L \circ F(\eta) $ depends non-linearly on $\eta$
and that, recalling \eqref{def:f}, 
$ F(\eta)[\phi] = 
F(\eta) [\Pi \phi ] $
depends only on the component $ \Pi \phi \in  \cH^{\sigma,s,a}$ of $ \phi $
defined in \eqref{coplusH}, 
for the presence of the derivatives $ \pa_{y},  \pa_{yy}, \nabla\pa_{y} $.   
In the next lemma we study the regularity of the nonlinear map $\eta \mapsto F(\eta)$.
\begin{lem} 
\label{lem:non}
Let $\sigma \geq 0$,  $s + \frac12, s_0 \in \bN$  with 
$ s-\frac32  \geq s_0 > \frac{d+1}{2} $. 
There exists $\e_0 :=\e_0(s) >0$ such that the nonlinear map
\begin{alignat*}{2}
F : H^{\sigma,s}&\cap B^{\sigma, s_0+\frac32}(\e_0) &&\to \cL\big(\bC \oplus \cH^{\sigma, s+\frac12, a}, \cH^{\sigma, s-\frac32,a} \big) \, , \\
&\; \eta &&\mapsto \qquad
\big\{\, \phi \mapsto F(\eta)[\phi]\, \big\}  \, , 
\end{alignat*}
defined in \eqref{def:f} is analytic and satisfies the tame estimate
\begin{equation}\label{stimalineare}
\|F(\eta)[\phi]\|_{\sigma,s-\frac32,a} \leq C(s) \big( \|\eta\|_{H^{\sigma,s_0+\frac32}} \|\Pi\phi\|_{\sigma,s+\frac12,a} + \|\eta\|_{H^{\sigma,s}} \|\Pi\phi\|_{\sigma,s_0+2,a}  \big) \, .
\end{equation}
\end{lem}
\begin{proof}
We write $F(\eta)[\phi]$ in \eqref{def:f} as 
$$
F(\eta)[\phi] =
 \cF_1[\alpha(\eta),  \phi] 
 +
  \cF_2[\beta(\eta),\phi]
  +
   \sum_{j=1}^d \cF_{3j}[\gamma_j(\eta), \phi] 
   +
    \cF_4[\delta(\eta),\phi] 
$$
with   bilinear maps 
\begin{align*}
\cF_1[g, \phi]:= g \pa_y^2 \phi \, , \quad 
\cF_2[g, \phi]:= g \Delta \phi \, , \quad 
\cF_{3j}[g, \phi]:=  g \pa_{x_j} \pa_y \phi \, , \quad 
\cF_4[g,\phi]:= g \pa_y \phi  \, .
\end{align*}
  In view of \eqref{algebrassa}, Lemma \ref{lem:prop} and \eqref{coplusH}, each of these maps is   bounded $\cH^{\sigma, s-\frac32, a} \times (\bC \oplus \cH^{\sigma, s+\frac12,a}) \to \cH^{\sigma, s-\frac32,a}$ and any $ \cF \in \{\cF_1, \cF_2, \cF_{3j}, \cF_4 \} $ 
  fulfills the tame  estimates
\begin{equation}
\label{Fi}
\| \cF [g, \phi] \|_{\sigma, s-\frac32,a} \lesssim_s  \| g \|_{\sigma, s-\frac32,a}
\|\Pi \phi \|_{\sigma, s_0+2,a} + \| g \|_{\sigma, s_0,a}
\|\Pi \phi \|_{\sigma, s+\frac12,a}   \, . 
\end{equation}
We claim that  the maps
\begin{equation}\label{mappeanaliti}
\begin{aligned}
& H^{\sigma, s}  \cap B^{\sigma,s_0+\frac12}(\e_0) \to \cH^{\sigma, s-\frac12, a} \, , 
\quad \eta \mapsto \alpha(\eta) \, , \beta(\eta) \, , \gamma_j(\eta) \, , \quad j = 1, \ldots, d \, , \\
& H^{\sigma, s} \cap B^{\sigma,s_0+\frac32}(\e_0)  \to \cH^{\sigma, s-\frac32, a} \, , \quad \eta \mapsto \delta(\eta) \, , 
\end{aligned}
\end{equation}
are analytic and, for any $s \geq s_0+\frac32 $, $ j=1,\dots, d $,  
\begin{equation}\label{midestimate}
\|\alpha(\eta) \|_{\sigma,s-\frac12,a},\, \|\beta(\eta) \|_{\sigma,s-\frac12,a},\, \|\gamma_j(\eta) \|_{\sigma,s-\frac12,a},\, \|\delta(\eta) \|_{\sigma,s-\frac32,a} \leq C(s) \|\eta\|_{H^{\sigma,s}} \, .
\end{equation} 
It is clear that these properties, together with 
\eqref{Fi},
 imply the Lemma. 

Let us consider first $\alpha(\eta)$, defined in \eqref{f}, which we rewrite as
$$
\alpha(\eta) = \Big(1- \frac{1}{\pa_y \rho(\eta)}\Big)+ \Big(1- \frac{1}{\pa_y \rho(\eta)}\Big)|\nabla \rho(\eta)|^2 -  |\nabla \rho(\eta)|^2 \, . 
$$
We first prove that   $\eta \mapsto 1- \frac{1}{\pa_y \rho(\eta)}$ is analytic
as a  map 
$H^{\sigma, s} \cap B^{\sigma,s_0}(\e_0)  \to \cH^{\sigma, s-\frac12, a}$. 
We first note that  Lemma \ref{lem:und.varphi} implies 
\begin{equation}\label{boundpay}
\|  \pa_y \rho(\eta)  - 1\|_{\sigma, s-\frac12 , a} 
=  \| e^{y|D|} |D| \eta \|_{\sigma, s-\frac12, a} \lesssim_s \| \eta \|_{H^{\sigma,s}} \, .  
\end{equation}
Then by \eqref{potenze} and \eqref{boundpay} the series
\begin{equation}
\label{neu.rho}
1- \frac{1}{\pa_y \rho(\eta)}  = 
-  \sum_{j \geq 1}   (  1- \pa_y \rho(\eta) )^j 
\end{equation}
is bounded by
\begin{align}
\Big\|  \frac{1}{\pa_y \rho(\eta)}-1 \Big\|_{\sigma, s-\frac12, a}
&\leq \|\pa_y \rho(\eta)-1  \|_{\sigma,s-\frac12,a} \sum_{j \geq 1}  \big(2C_s \|\pa_y \rho(\eta)-1\|_{\sigma,s_0,a}\big)^{j-1} \leq C(s) \|\eta  \|_{H^{\sigma,s} }  \label{fracest}
\end{align}
provided 
$  \| \eta \|_{H^{\sigma,s_0+ \frac12}} < \e_0(s)$ is small enough.
The  series \eqref{neu.rho} of analytic functions in uniformly 
convergent in $\cH^{\sigma, s-\frac12,a}$ on the domain $\eta \in 
B^{\sigma,s}(R)  \cap B^{\sigma,s_0}(\e_0)$, $ \forall R > 0 $, 
thus it defines 
an analytic  map
on $ H^{\sigma, s}  \cap B^{\sigma,s_0}(\e_0) $.
Moreover the linear map  $ \eta \mapsto \nabla \rho(\eta) =
e^{y|D|}  \nabla \eta  $ is, by Lemma \ref{lem:und.varphi},  
bounded between $ H^{\sigma, s} \to \cH^{\sigma, s-\frac12,a} $.
Therefore  $\alpha(\eta)$ is the product of analytic functions  $H^{\sigma, s} \cap B^{\sigma,s_0+\frac12}(\e_0) \to \cH^{\sigma, s-\frac12,a}$, and using the tame estimate 
\eqref{algebrassa} we get  \eqref{midestimate}.

The analyticity  and the estimates of the functions
$\eta \mapsto \beta(\eta), \gamma_j (\eta)$, $ j = 1, \ldots, d$
stated in \eqref{mappeanaliti} follow similarly.
Finally consider  $\delta(\eta)$ in \eqref{f}.
The  biggest loss of derivatives follows from  the linear  maps
$\eta \mapsto \Delta \rho(\eta) $, $ \pa_y^2\rho(\eta) $, $ \nabla \pa_y\rho(\eta) $
which, by Lemmata \ref{lem:prop} and \ref{lem:und.varphi}, 
are bounded between $ H^{\sigma, s} \to \cH^{\sigma, s-\frac32,a} $.
Moreover $\delta(\eta)$ satisfies the estimate
$\| \delta(\eta) \|_{H^{\sigma, s-\frac32,a}} \leq C(s, \| \eta\|_{H^{\sigma, s_0+\frac32}} ) \, \| \eta\|_{H^{\sigma, s}} $
for any $ s -\frac32 \geq s_0$.
\end{proof}

\begin{proof}[Proof of Lemma \ref{lem:lasol}] For any $s\geq s_0+\frac32$ such that $s+\frac12\in\bN$,
by Lemmata \ref{lem:ell5} and \ref{lem:non}, the map 
$$
\eta \mapsto P(\eta):= L\circ  F(\eta) \, , \quad 
H^{\sigma,s}\cap B^{\sigma, s_0+\frac12}(\e_0) \to \cL(\bC \oplus \cH^{\sigma, s+\frac12, a}) \, , 
$$ 
is analytic and, for positive constants $ C(s) \geq C'(s_0) > 0 $,  
 in view of \eqref{coplusH}, 
\begin{equation}\label{indbase}
\begin{aligned}
&
\| P(\eta)[\phi] \|_{\sigma, s+\frac12, a} \leq C(s) \, \big(  \| \eta \|_{H^{\sigma, s_0+\frac32}}\|\Pi\phi \|_{\sigma, s+\frac12, a}+\| \eta \|_{H^{\sigma, s}}\|\Pi\phi \|_{\sigma, s_0+2, a} \big) \\
& \| P(\eta)[\phi] \|_{\sigma, s_0+2, a} \leq C'(s_0) \,  
\| \eta \|_{H^{\sigma, s_0+\frac32}}  \|\Pi\phi \|_{\sigma, s_0+2, a}  \, . 
\end{aligned}
\end{equation}
We claim that, for any $\eta \in H^{\sigma,s}\cap B^{\sigma, s_0+\frac32}(\e_0) $  
and $  \e_0 (s) > 0 $ small enough, the operator $\uno-P(\eta) $ is invertible in $\cL(\bC \oplus \cH^{\sigma, s+\frac12, a})$ and the inverse map
 \begin{equation}\label{Neumannseries}
 \eta \mapsto (\uno - P(\eta))^{-1} = \sum_{j=0}^\infty P(\eta)^j[\phi] \, , \quad H^{\sigma,s}\cap B^{\sigma, s_0+\frac32}(\e_0) \to \cL(\bC \oplus \cH^{\sigma, s+\frac12, a})  \, , 
 \end{equation}
 is analytic. 
 As each $\eta \mapsto P(\eta)^j$ is analytic $ H^{\sigma,s}\cap B^{\sigma, s_0+\frac12}(\e_0) \to \cL(\bC \oplus \cH^{\sigma, s+\frac12, a})$, 
the claim follows by proving that the series \eqref{Neumannseries} converges uniformly  in 
$\cL(\bC \oplus \cH^{\sigma, s+\frac12, a})$ for 
$\eta \in B^{\sigma,s}(R)\cap B^{\sigma, s_0+\frac32}(\e_0) $ for any $R >0$.
By \eqref{indbase}  we have, for any $ j \in \bN $, 
\begin{equation}\label{indbassa}
\| P(\eta)^j[\phi] \|_{\sigma, s_0+2, a} \leq \big(C'(s_0) \| \eta \|_{H^{\sigma, s_0+\frac32}}\big)^j \|\Pi\phi \|_{\sigma, s_0+2, a} \, , 
\end{equation}
and, by induction, we prove that 
\begin{equation}\label{indalta}
\| P(\eta)^j[\phi] \|_{\sigma, s+\frac12, a} \leq C(s)^j \|\eta \|_{H^{\sigma, s_0+\frac32}}^{j-1}  \big(  \| \eta \|_{H^{\sigma, s_0+\frac32}}\|\Pi\phi \|_{\sigma, s+\frac12, a}\!+\!j\| \eta \|_{H^{\sigma, s}}\|\Pi\phi \|_{\sigma, s_0+2, a} \big) \, . 
\end{equation}
Indeed, for $j=1$ this is \eqref{indbase}. 
Then assuming that  \eqref{indalta} holds for $j$, we get 
\begin{equation*}
\| P(\eta)^{j+1}[\phi] \|_{\sigma, s+\frac12, a} \stackrel{\eqref{indbase}}{\leq} C(s) \, \big(\underbrace{\|\eta \|_{H^{\sigma, s_0+\frac32}}\|P(\eta)^j[\phi] \|_{\sigma, s+\frac12, a}}_{=:A}+\underbrace{\| \eta \|_{H^{\sigma, s}}\| P(\eta)^j[\phi] \|_{\sigma, s_0+2, a}}_{=:B} \big).
\end{equation*}
By the inductive hypothesis the first term is bounded by
\begin{equation*}
A \leq C(s)^{j} \|\eta \|_{H^{\sigma, s_0+\frac32}}^{j}  \, \big(  \| \eta \|_{H^{\sigma, s_0+\frac32}}\|\Pi\phi \|_{\sigma, s+\frac12, a}+j\| \eta \|_{H^{\sigma, s}}\|\Pi\phi \|_{\sigma, s_0+2, a} \big) \, , 
\end{equation*}
whereas, by \eqref{indbassa}, 
\begin{equation*}
B\leq (C'(s_0) \| \eta \|_{H^{\sigma, s_0+\frac32}})^{j} \| \eta \|_{H^{\sigma, s}} \|\Pi\phi \|_{\sigma, s_0+2, a} \, ,
\end{equation*}
and we deduce, as $ C'(s_0) \leq C(s) $,  that 
\begin{equation*}
\| P(\eta)^{j+1}[\phi] \|_{\sigma, s+\frac12, a} \leq C(s)^{j+1} \|\eta \|_{H^{\sigma, s_0+\frac32}}^{j}  \, \big(  \| \eta \|_{H^{\sigma, s_0+\frac32}}\|\Pi\phi \|_{\sigma, s+\frac12, a}+(j+1)\| \eta \|_{H^{\sigma, s}}\|\Pi\phi \|_{\sigma, s_0+2, a} \big) \, ,
\end{equation*}
which proves \eqref{indalta} at the step $ j+1 $. 

By \eqref{indalta}, 
the series in \eqref{Neumannseries} is bounded by 
\begin{align}
\|(\uno-P(\eta))^{-1}[\phi]\|_{\sigma,s+\frac12,a} &\leq \sum_{j \geq 0} \| P(\eta)^j[\phi] \|_{\sigma,s+\frac12,a}  \notag \\ \notag
&\leq \|\Pi\phi \|_{\sigma, s+\frac12, a} + 
\sum_{j \geq 1} C(s)^j \|\eta \|_{H^{\sigma, s_0+\frac32}}^{j-1}\   \| \eta \|_{H^{\sigma, s_0+\frac32}}\|\Pi\phi \|_{\sigma, s+\frac12, a}  \\
&\qquad \qquad \qquad \, +\sum_{j \geq 1} C(s)^j \|\eta \|_{H^{\sigma, s_0+\frac32}}^{j-1}\,j\,\| \eta \|_{H^{\sigma, s}}\|\Pi\phi \|_{\sigma, s_0+2, a} \notag \\
&\leq 2\|\Pi\phi\|_{\sigma,s+\frac12,a}+
C \|\eta\|_{H^{\sigma,s}} \|\Pi\phi\|_{\sigma,s_0+2,a} 
\label{finalest}
\end{align}
provided 
$ \| \eta \|_{H^{\sigma, s_0+\frac32}} < \e_0(s)$ is sufficiently small. 
In particular this shows the claim on the uniform convergence of the series on $B^{\sigma,s}(R)\cap B^{\sigma, s_0+\frac32}(\e_0) $ for any $R >0$.

The analytic map 	
\begin{align*}
U: H^{\sigma,s}\cap B^{\sigma, s_0+\frac32}(\e_0) \longrightarrow \cL \big( \bC \oplus \cH^{\sigma, s+\frac12,a}\big) \, , \quad 
 U(\eta)[ \phi]:= (\uno - L \circ F(\eta) )^{-1}\big[L\circ F(\eta)[\phi]\big] \, , 
\end{align*}
defines the unique solution $u = U(\eta)[\phi]$ 
of  \eqref{ell10} and, consequently, of system \eqref{ell4}.
By \eqref{finalest} and \eqref{indbase} we deduce
\eqref{tameu}. This  proves Lemma \ref{lem:lasol}.
\end{proof}

\begin{proof}[Proof of Proposition \ref{prop:varphi}] 
It follows with $ \varphi = \Psi(\eta)[\psi]= e^{y|D|}\psi + U (\eta)[e^{y|D|}\psi] $, see
\eqref{eq:varphi}, \eqref{ell3} and Lemma \ref{lem:lasol}. 
\end{proof}

\section{Analyticity of the Stokes wave}\label{sec:Sto}

In this section we prove Theorem \ref{LeviCivita}. 
With the aid of the analyticity result of Theorem \ref{thm:DN}, the bifurcation 
proof is classical.  
We report it for completeness. 
It is based on the application of the analytic Crandall-Rabinowitz  Theorem \ref{thm:CrandallRabinowitz} below.  For the proof
we refer e.g to  
\cite{BuT}, and  Theorem 4.1 in Chap. 5 of \cite{AP} for its smooth version.
\begin{teo}[Crandall-Rabinowitz bifurcation Theorem]\label{thm:CrandallRabinowitz}
Let $X,Y$ be Banach spaces and $  U \subset X$ be an open neighbourhood of 0. Let  $F:U\times \bR \to Y$, $ F(u, c ) $, be an analytic map satisfying 
 $F(0,c)=0$ for any $c \in \bR$. Let $c^*$ be such  that 
 $ L := \de_u F(0,c^*) \in {\mathcal L} (X,Y)$ is not invertible and 
\begin{enumerate}
\item 
$ \text{Ker} ( L )   = \textup{span} \langle u^* \rangle $, $ u^* \in X $, is  $ 1 $-dimensional; 
\item the range  $ R := \text{Rng} (L ) $  is closed and $\textup{codim } R = 1 $; 
\item {(\it {transversality})} 
$ \; \pa_c \de_u F(0,c^*)[u^*] \notin R   \, . $
\end{enumerate}
Then there exist $\e_*  >0$ and an analytic function 
$$
(-\e_*,\e_*) \to   U\times \bR , \ \ \e \mapsto (u_\e, c_\e) \, , \   \quad \quad 
 u_\e = \e u^* + O(\e^2) \, , \ \  c_\e = c^* + O(\e) \, ,
  $$ 
such that 
  $F(u_\e, c_\e) = 0$ for any $| \e | < \e_* $.
\end{teo}

Theorem \ref{LeviCivita}  is proved by applying Theorem \ref{thm:CrandallRabinowitz} to 
the nonlinear operator 
\begin{equation}\label{def:F}
\begin{aligned}
& F:\; \big( H^{\sigma,s}_{\mathtt{ev}_0} \cap B^{\sigma,s_0}(\e_0) \big)
\times  H^{\sigma,s}_{\mathtt{odd}}   
\times \bR \longrightarrow   H^{\sigma,s-1}_{\mathtt{odd}} 
\times  H^{\sigma,s-1}_{\mathtt{ev}_0}\, , \ \sigma \geq 0 \, , \  s > 5/ 2  \, , \\
& \qquad  F(\eta,\psi,c):= \Big(c\eta_x+G(\eta)\psi \ , \ 
c\psi_x - g \eta - \dfrac{\psi_x^2}{2} + 
\dfrac{1}{2(1+\eta_x^2)} \big( G(\eta) \psi + \eta_x \psi_x \big)^2   \Big) 
\end{aligned}
\end{equation}
where 
$ H^{\sigma,s}_{\mathtt{ev}_0} $, respectively $ H^{\sigma,s}_{\mathtt{odd}} $, denote the  space of even, respectively odd, and average-free real valued functions in $ H^{\sigma,s}$ defined in \eqref{def:Hsigmas},
and $ \e_0 := \e_0(\sigma,s,s_0) > 0 $ is provided by Theorem \ref{thm:DN}. Note that a 
real  function $ (\eta, \psi) \in H^{\sigma,s}_{\mathtt{ev}_0} \times H^{\sigma,s}_{\mathtt{odd}}  $ admits a Fourier series expansion 
\begin{equation}\label{svifou}
\vet{\eta(x)}{\psi(x)} = \sum_{k \geq 1} \vet{\eta_k \cos(kx)}{\psi_k \sin(kx)} \quad \text{with norm} \quad 
\| (\eta,\psi) \|_{H^{\sigma,s}}^2 \simeq \sum_{k \geq 1} 
 e^{2 \sigma |k|} \langle k \rangle^{2s} (\eta_k^2 + \psi_k^2)  \, . 
\end{equation}
The fact that the nonlinear operator $ F  $ in \eqref{def:F} maps 
a pair of functions $ (\eta, \psi )$ which are odd/even in $ x $
 into a pair of functions which are even/odd in $ x $ is verified thanks to the reversibility property \eqref{DN-rev}.   Moreover, the second component of $ F $
 has zero average thanks to the following  lemma. 
\begin{lem}\label{lem:averagefree}
 Let $G ( \eta) $ be the Dirichlet-Neumann operator defined in \eqref{DN1}. Then  
\begin{equation}\label{DNinf}
\int_\bT -  \frac12 \psi_x^2  +  \dfrac{1}{2(1+\eta_x^2)} \big( G(\eta) \psi + \eta_x \psi_x \big)^2
\, \de x  = 0 \, . 
\end{equation}
\end{lem}
\begin{proof} 
By \eqref{DN-inf}, the kinetic energy
$ K(\eta, \psi) = \frac12 (\psi, G(\eta) \psi)_{L^2}  $ in \eqref{kine} satisfies
$ K(\eta+m,\psi) = K(\eta,\psi) $ for any $ m \in \bR $.  
Thus 
$$
0 = \frac{\de}{\de m} K(\eta+m,\psi)
	= \de_{\eta} K(\eta,\psi)[1] = (\nabla_\eta K(\eta, \psi), 1)_{L^2} 
	= \int_{\bT} \nabla_\eta K(\eta, \psi) \, \de x \, .   
$$
In view of  \eqref{Kform}, the identity  \eqref{DNinf} is proved. 
\end{proof}

\smallskip
We now start verifying the assumptions of the Crandall-Rabinowitz Theorem \ref{thm:CrandallRabinowitz}.
First, by Theorem \ref{thm:DN}, the nonlinear operator 
 $ F $ defined in \eqref{def:F} is {\it analytic}. Moreover, by inspection, 
$$
F(0,0,c) = 0 \, , \quad \forall c \in \bR \, . 
$$
The possible bifurcation values of non-trivial solutions of $ F(\eta, \psi, c ) = 0 $ 
are those speeds $ c $ such that 
the  linearized operator 
\begin{equation}\label{linF}
\begin{aligned}
\de_{(\eta,\psi)} F(0,0,c):\; H^{\sigma,s}_{\mathtt{ev}_0} \times H^{\sigma,s}_{\mathtt{odd}} \to& H^{\sigma,s-1}_{\mathtt{odd}} \times  H^{\sigma,s-1}_{\mathtt{ev}_0} \, , \quad
 \vet{\hat\eta}{\hat\psi} \mapsto& \begin{bmatrix} c\pa_x & |D| \\ -g & c \pa_x \end{bmatrix}\vet{\hat\eta}{\hat\psi} \, , 
\end{aligned}
\end{equation}
has a nontrivial kernel. In the next lemma we characterize such values.
\begin{lem}\label{lem:ker}
{\bf (Bifurcation speeds)}
The kernel of $\de_{(\eta,\psi)} F(0,0,c)$ in \eqref{linF} is nontrivial if and only if 
\begin{equation}\label{ilc} 
c= \pm \sqrt{\frac{g}{k}} \qquad \text{for  some} \  k \in \bN \, . 
\end{equation}
For any  $ k \in \bN $, the Kernel of $ L := \de_{(\eta,\psi)} F(0,0, c^*_k) $,
where we set $ c^*_k := \sqrt{\frac{g}{k}}$, 
is one dimensional and 
\begin{equation}\label{kernelL}
\text{Ker}(L)= \langle u^* \rangle  \quad \text{with} \quad 
u^* := \vet{\sqrt{k}\cos(kx)}{\sqrt{g}\sin(kx)} \, . 
\end{equation}
\end{lem}

\begin{proof}
By the Fourier expansion 
  \eqref{svifou}, it results that the kernel of $\de_{(\eta,\psi)} F(0,0,c) $ is nontrivial if
and only if at least one of the matrices
$ \begin{bmatrix} - ck   & k \\ -g &  c k \end{bmatrix}$, $k \in \bN$, has zero determinant. This is  verified provided  $c^2 k = g $ for some  $k \in \bN $, i.e. \eqref{ilc} holds.
In addition, a vector 
$ \vet{\eta(x)}{\psi(x)} = \sum_{j \geq 1} \vet{\eta_j \cos(jx)}{\psi_j \sin(jx)}$  
belongs to the Kernel of 
$  \de_{(\eta,\psi)} F(0,0,c_k^* ) $ if and only if  
\begin{equation}
\label{CR1.sys}
\begin{bmatrix}
- c^*_k j & j \\
-g & c^*_k j 
\end{bmatrix}\vet{\eta_j}{\psi_j} =  0 \, , \quad \forall j \geq 1 \, . 
\end{equation}
If  $j \neq k$ then 
\begin{equation}
\label{det.ck}
\det \begin{bmatrix}
- c^*_k j & j \\
-g & c^*_k j 
\end{bmatrix}  = - (c^*_k)^2 j^2 + g j  
= j^2 \big( (c^*_j)^2 - (c^*_k)^2\big) \neq 0 \, ,  
\end{equation}
since  the map $k \mapsto (c^*_k)^2 = g / k $ is  injective on $\bN$. 
Hence $\eta_j = \psi_j = 0$ for any $j \neq k$. On the other hand, if  $j = k$
then \eqref{CR1.sys} is solved provided  $\sqrt{g} \eta_k =\sqrt{k}\psi_k$, proving 
\eqref{kernelL}. 
\end{proof}

We apply Theorem  \ref{thm:CrandallRabinowitz} with 
$ c^*_k := \sqrt{\frac{g}{k}}$. By Lemma \ref{lem:ker}
 assumption 1 holds. 
The next lemma verifies the assumptions 2)-3). 
\begin{lem}\label{lem:CrandallRabinowitzHP}
The range  $ R:=\text{Rng} L$,  $ L = \de_{(\eta,\psi)} F(0,0, c^*_k) $,  is 
\begin{equation}\label{formaR}
R= \left\lbrace
 \vet{f}{g} \in H^{\sigma,s-1}_{\mathtt{odd}}(\bT)\times  H^{\sigma,s-1}_{\mathtt{ev}_0}(\bT) 
  \, \colon 
  \ \ 
  \vet{f(x)}{g(x)} =  \vet{f_k \sin(kx)}{ c_k^* f_k \cos( kx)}  + 
  \sum_{j\geq 1, j \neq k} \vet{f_j \sin(jx)}{g_j \cos(jx)}   
\right\rbrace \, . 
\end{equation}
 In particular $R$ is closed and $\mathrm{codim}\, R =1 $.
 
The vector $(\pa_c \de_{(\eta,\psi)} F)(0,0, c^*_k)\vet{\sqrt{k}\cos(kx)}{\sqrt{g}\sin(kx)}$ does not belong to $R$.
\end{lem}

\begin{proof}
A vector  $\vet{f}{g}\in H^{\sigma,s-1}_{\mathtt{odd}}(\bT)\times  H^{\sigma,s-1}_{\mathtt{ev}_0}(\bT) $ belongs to $ R $ if and only if 
there is $\vet{\eta}{\psi} \in H^{\sigma,s}_{\mathtt{ev}_0} \times  H^{\sigma,s}_{\mathtt{odd}}$ 
such that, recalling  \eqref{linF} and \eqref{svifou},  
\begin{equation}
\label{CR2.sys}
\begin{bmatrix}
- c^*_k j & j \\
-g & c^*_k j 
\end{bmatrix}\vet{\eta_j}{\psi_j} = 
\vet{f_j}{g_j} \quad \forall j \geq 1  
\quad \text{where} \quad 
\vet{f(x)}{g(x)} = \sum_{j\geq 1} \vet{f_j \sin(jx)}{g_j \cos(jx)}  \, . 
\end{equation}
For any $j \neq k$, by \eqref{det.ck}, system \eqref{CR2.sys} has the unique solution 
\begin{equation}
\label{sol.1}
\eta_j = \frac1g \frac{\sqrt{k}}{k-j}(\sqrt{g}f_j-\sqrt{k}g_j) \, , \quad \psi_j = \frac1{j\sqrt{g}}\frac{\sqrt{k}}{k-j}(\sqrt{kg}f_j-jg_j) \, .
\end{equation}
 If  $j = k$, the system \eqref{CR2.sys}  is solvable if and only if
\begin{equation}
\label{fkgk}
\sqrt{g} f_k = \sqrt{k} g_k
\end{equation}
and  a solution is  $\eta_k = - \frac{1}{\sqrt{kg}}f_k$, $\psi_k = 0$.
By  \eqref{sol.1} we deduce that 
$|\eta_j|,  |\psi_j| \leq \frac{C_k}{j} (|f_j| + |g_j|)$, for any $ j \in \bN \setminus \{k \}$, implying 
that $ (\eta, \psi) \in H^{\sigma,s} $, actually
$$
\| \eta \|_{H^{\sigma,s}},\; \| \psi \|_{H^{\sigma,s}}  \leq C_k (\|f\|_{H^{\sigma,s-1}}+\|g\|_{H^{\sigma,s-1}}) \, . 
$$ 
In conclusion, 
the range $ R $ of $ L $ has the form \eqref{formaR},  
by \eqref{fkgk} and $ c^*_k = \sqrt{g/k} $. 

Finally differentiating \eqref{linF} one computes
 $$
 (\pa_c \de_{(\eta,\psi)} F)(0,0,c^*_k)\vet{\sqrt{k}\cos(kx)}{\sqrt{g}\sin(kx)}= {\scriptsize \begin{bmatrix} \pa_x \ \ 0 \\ 0 \ \ \pa_x \end{bmatrix}}\vet{\sqrt{k}\cos(kx)}{\sqrt{g}\sin(kx)} =\vet{-k^{\frac32}\sin(kx)}{k\sqrt{g}\cos(kx)} $$
 which  does not belong to the range $R$ in \eqref{formaR}.
\end{proof}

All the  assumptions of the Crandall-Rabinowitz Theorem are verified,  proving Theorem \ref{LeviCivita}.

\appendix

\section{Basic properties of the Dirichlet-Neumann operator}
\label{app:DN}

The linear Dirichlet-Neumann operator $ G(\eta) $ defined in \eqref{DN1}
is self-adjoint with respect to the $ L^2 $ scalar product, 
$$
\big( G(\eta)\psi_1, \psi_2 \big)_{L^2} = 
\int_{{\cal D}_\eta} \nabla \Phi_1 \cdot \nabla \Phi_2 
\, \de x = \big( G(\eta)\psi_2, \psi_1 \big)_{L^2} \, , 
$$
where $ \Phi_1 $ and $ \Phi_2 $ are the harmonic functions associated to $ \psi_1, \psi_2 $ as in \eqref{Phi}.
Thus $ G(\eta) $ is semi-positive definite
\begin{equation}\label{DNSP}
\big( G(\eta)\psi, \psi \big)_{L^2} = \int_{{\cal D}_\eta} |\nabla \Phi |^2 \de x \geq 0 \, , 
\end{equation}
and its kernel contains only the constant functions, $ G(\eta)[1] = 0 $.
In particular 
\eqref{DNSP} implies also the unicity of the solutions of \eqref{Phi}. 

We list other classical algebraic properties of the Dirichlet-Neumann used in the paper. 

\begin{lem} 
The Dirichlet-Neumann $ G(\eta) $ in \eqref{DN1} is:
\\[1mm] 
(i) invariant under space translations
 \begin{equation}\label{DN-inv}
 \tau_\theta G(\eta)\psi = G( \tau_\theta \eta)[ \tau_\theta \psi] \, , \quad 
 \tau_\theta u (x) := u (x + \theta ) \, , \ \  \forall \theta \in \bR^d \, ;
 \end{equation}
(ii) invariant under the reflection at the origin, namely 
\begin{equation}\label{DN-rev}
G(  \eta^\vee ) [ \psi^\vee ] = \left( G(\eta) [\psi ] \right)^\vee  \quad
\text{where} \quad  f^\vee (x) := f (-  x)  \, ; 
\end{equation} 
(iii) constant along vertical translations, i.e. 
\begin{equation}\label{DN-inf}
 G(\eta + m) = G(\eta) \, , \quad \forall  m \in \mathbb{R} \, . 
 \end{equation} 
 \end{lem}
 
 \begin{proof}
Let us prove \eqref{DN-inv}. 
Let $ \Phi $ be the solution of \eqref{Phi}.  
For any $ \theta \in \bR^d $ the harmonic function 
 $$
 \Phi_{\theta} (x,y) := \Phi(x+ \theta,y)  \quad 
 \forall (x,y) \in {\cal D}_{\tau_\theta \eta} = \{ y < \eta (x+ \theta )\}
 $$
 solves
 $$
\Delta_{x,y} \Phi_{\theta} = 0 \  \text{in} \ {\cal D}_{\tau_\theta \eta} \, ,
\quad 
\Phi_\theta (x,  \tau_\theta \eta(x) ) = \tau_\theta \psi(x)  \, , \quad 
\pa_y \Phi_\theta (x,y) \to 0  \mbox{ as } y \to - \infty \, .
$$ 
Therefore, by \eqref{DN1},  
$$
\begin{aligned}
G( \tau_\theta \eta )[\tau_\theta \psi ] 
& = 
(\pa_y \Phi_\theta) (x, \tau_\theta \eta) - 
(\nabla \tau_\theta \eta)(x) \cdot (\nabla \Phi_\theta )(x,\tau_\theta \eta) \\
& = 
(\pa_y \Phi) (x+\theta, \eta(x+\theta)) - 
(\nabla \eta)(x+\theta) \cdot (\nabla \Phi)(x+\theta, \eta(x+\theta)) 
= \tau_\theta  G(\eta )[\psi ]
\end{aligned}
$$
proving \eqref{DN-inv}.  To prove \eqref{DN-rev}, consider the harmonic function 
 $$
 \Phi^\vee (x,y) := \Phi(-x,y)  \quad 
 \forall (x,y) \in {\cal D}_{\eta^\vee} = \{ y < \eta^	\vee (x) \}
 $$
which solves
   $$
\Delta_{x,y} \Phi^\vee = 0 \  \text{in} \ {\cal D}^\vee_{\eta} \, ,
\quad 
\Phi^\vee  (x, \eta^\vee (x)) = \psi^\vee (x)  \, , \quad 
\pa_y \Phi^\vee  (x,y) \to 0  \mbox{ as } y \to - \infty \, .
$$ 
Therefore  \eqref{DN-rev} follows by  
$$
\begin{aligned}
G(  \eta^\vee )[  \psi^\vee ] 
& = 
(\pa_y \Phi^\vee) (x, \eta^\vee (x) ) - 
(\nabla \eta^\vee)(x) \cdot (\nabla \Phi^\vee )(x, \eta^\vee (x)) \\
& = 
(\pa_y \Phi) (-x, \eta (-x) ) - 
(\nabla \eta)(-x) \cdot (\nabla \Phi )(-x, \eta (-x) ) = G(\eta )[\psi ](-x) \, .
\end{aligned}
$$
For any $ m  \in \bR $ the harmonic function 
 $$
 \Phi_m (x,y) := \Phi(x,y-m)  \quad 
 \forall (x,y) \in {\cal D}_{\eta + m} = \{ y < \eta (x) + m \}
 $$
 solves
  $$
\Delta_{x,y} \Phi_m = 0 \  \text{in} \ {\cal D}_{\eta + m} \, ,
\quad 
\Phi_m (x, \eta(x) + m ) = \psi(x)  \, , \quad 
\pa_y \Phi_m  (x,y) \to 0  \mbox{ as } y \to - \infty \, .
$$ 
Therefore, by \eqref{DN1},  
$$
\begin{aligned}
G(  \eta + m )[  \psi ] 
& = 
(\pa_y \Phi_m) (x, \eta (x) + m ) - 
(\nabla \eta)(x) \cdot (\nabla \Phi_m )(x, \eta (x) +m) \\
& = 
(\pa_y \Phi) (x, \eta (x) ) - 
(\nabla \eta)(x) \cdot (\nabla \Phi )(x, \eta (x) ) = G(\eta )[\psi ] 
\end{aligned}
$$
proving  \eqref{DN-inf}. 
\end{proof}

\section{Functional spaces}
\label{App1}

We collect in this Appendix some properties of 
the  function spaces $ H^{\sigma,s} $ and $ {\mathcal H}^{\sigma,s,a} $.

\subsection{The spaces $ H^{\sigma,s} $} \label{App1A}

We first note the following characterization of the spaces $ H^{\sigma, s} $. 

\begin{lem}\label{lem:equinorms}
{\bf (Characterization of $ H^{\sigma, s} $)} 
The space 
$H^{\sigma, s}(\bT^d)$, $ \sigma > 0 $,  coincides with the periodic functions 
$ u(x) $ 
which admit an extension $ u(z) $ in the complex strip 
$$ 
\bT^d_\sigma := \bT^d + \im [-\sigma,\sigma]^d =
\Big\{ z = x + \im y \, \colon  \, \, x \in \bT^d \, , \; y\in \bR^d,\, \ 
|y|_{\infty} := \max\{|y_1|, \ldots, |y_d| \} \leq \sigma \Big\} \, , 
$$
which is analytic in $ |y|_{\infty} < \sigma $, and 
whose traces at the boundaries  $ u( \cdot  + \im y ) $,  $ |y|_\infty = \sigma $,  
belong to   the Sobolev space $H^s := H^s(\bT^d)$, with equivalence of the norms
\begin{equation}\label{uequiv}
\| u \|_{H^{\sigma, s}} \simeq_{d} \, \sup_{| y |_\infty \leq \sigma} 
\Big\{ \| u ( \cdot + \im y ) \|_{H^s} \Big \} \, .
\end{equation}
\end{lem}

\begin{proof}
Let  $ u (x) $ be a function in $ H^{\sigma, s}(\bT^d) $. 
For any $ z \in \bC^d $, $ z = x + \im y $, $|y|_\infty \leq \sigma $, 
we define its extension
$$
u(z) := \sum_{k \in \bZ^d} u_k \,  e^{\im k \cdot z} 
$$
which  is analytic for $ |y|_\infty < \sigma $.  
For any $ y \in \bR^d$ with $|y|_\infty \leq \sigma $, 
the Sobolev norm $ \| \ \|_{H^s} $ of the periodic function
\begin{equation}\label{uya}
x \mapsto u^{(y)} (x) := u( x + \im y ) =
\sum_{k \in \bZ^d} u_k \, e^{-  k \cdot y} e^{\im k \cdot x}  
\end{equation}
is bounded by 
\begin{align*}
\| u^{(y)} \|_{H^s}^2 = \sum_{k \in \bZd}
| u_k|^2 e^{- 2 k \cdot y } \langle k \rangle^{2s}
& \leq \sum_{k \in \bZd}
|u_k|^2 e^{2 |y|_\infty |k|_1} \langle k \rangle^{2s} \notag  \\
& \leq \sum_{k \in \bZd} | u_k|^2 e^{2 \sigma |k|_1} \langle k \rangle^{2s} =  
\| u \|_{H^{\sigma,s}}^2 \, . 
\end{align*}
Thus $ u^{(y)} ( \cdot ) $ belongs to $ H^s $ and 
 $ \| u^{(y)} \|_{H^s} \leq \| u \|_{H^{\sigma,s}} $.  
 \\[1mm]
In order to prove the equivalence \eqref{uequiv},  consider the partition of  $ \bZ^d $, 
$$
\bZ^d = \bigcup_{\vec{\e} \in \{ \pm 1 \}^d } \bZ^d_{\vec \e} \, , \quad  
\bZ^d_{\vec \e} := \left\{ k = (k_1, \ldots, k_d) \in \bZ^d \, : \, 
\begin{cases} 
k_j  > 0 \, ,  \quad \text{if} \ \e_j = - 1 \, ,  \\  
k_j \leq 0 \, , \quad \text{if} \ \e_j =  1\, , 
\end{cases}
\forall j = 1, \ldots, d 
\right\}  \, . 
$$
For any $ \vec \e = (\e_1, \ldots, \e_d) \in \{ \pm 1\}^d  $, the function
$ u^{(\sigma \vec \e)} $ defined as in \eqref{uya} satisfies 
\begin{align*}
\| u^{(\sigma \vec \e)}  \|_{H^s}^2 
= \sum_{k \in \bZd} | u_k|^2 
e^{- 2 \sigma  k \cdot \vec \e}  \langle k \rangle^{2s} 
& \geq  
\sum_{k \in \bZd_{\vec \e}} | u_k|^2 \langle k \rangle^{2s}
e^{2 \sigma (|k_1|+ \ldots + |k_d|)}  
\end{align*}
and therefore
$$
\|  u \|_{H^{\sigma,s}}^2  = 
\sum_{\vec \e \in \{ \pm 1\}^d} \sum_{k \in \bZd_{\vec \e}} | u_k|^2 \langle k \rangle^{2s}
e^{2 \sigma |k|_1} \leq 2^d \sup_{|y|_\infty = \sigma} \| u^{(y)}  \|_{H^s}^2  \, . 
$$
The equivalence \eqref{uequiv} is proved.
\end{proof}

The spaces $ H^{\sigma,s} $, $ s > d/ 2 $,
form an algebra with respect to the product of functions, 
and 
 the following more general  tame estimates hold.
  
\begin{lem}\label{lem:surfacingsubalgebra}
 {\bf (Tame)} Let  $\sigma \geq 0$ and $ s \geq \mathfrak{s}_0 > d/2 $. 
There exist positive constants $ C_{s, \mathfrak{s}_0} \geq 1 $ (non decreasing in $ s $) such that, for any 
$ f, g \in H^{\sigma, s}$, one has 
\begin{equation}\label{tameHs}
\| f g \|_{H^{\sigma,s}} \leq C_{s, \mathfrak{s}_0}\big( \| f \|_{H^{\sigma,s}} \| g \|_{H^{\sigma, \mathfrak{s}_0}} + \| f \|_{H^{\sigma, \mathfrak{s}_0}} \| g \|_{H^{\sigma,s}} \big) \, . 
\end{equation}
In particular, for any $j\geq 1$,
\begin{equation}\label{potHs}
\| f^j \|_{H^{\sigma, s}} \leq (2C_{s, \mathfrak{s}_0} \|f\|_{H^{\sigma, \mathfrak{s}_0}})^{j-1} \|f\|_{H^{\sigma, s}}\, .
\end{equation}
\end{lem}

\begin{proof}
The classical proof follows adapting the proof  of Lemma 4.5.1 in \cite{BB20}
and it  is quite similar to that of Lemma \ref{lem:tame}. So we omit it.  
Estimate \eqref{potHs} follows by induction from \eqref{tameHs} in the same way \eqref{potenze} descends from \eqref{algebrassa}. 
\end{proof}

\subsection{The spaces $\cH^{\sigma, s, a} $}\label{sec:A2}

\noindent
{\bf Proof of Lemma \ref{lem:trace}.}
For any  $u \in C^\infty_c(\bT^d \times \bR_{\leq 0})$, any $ y_0 \leq 0 $, 
 we have the inequality
$$
|  u_k(y_0) |^2 \leq 
2 \int_{-\infty}^0 | \pa_y  u_k(y) | \, |u_k(y)| \, \de y \, .
$$
Multiplying by  $\langle k \rangle^{2s}$ and using the elementary 
inequality
$2 \langle k \rangle^{2s} ab \leq \langle k \rangle^{2s-1} a^2 + \langle k \rangle^{2s+1} b^2 $, for any $a, b \geq 0 $, 
 we get that
$$
\begin{aligned}
\| u(\cdot, y_0) \|_{H^{\sigma,s}}^2 
& =
\sum_{k \in \bZd}   e^{2\sigma |k|_1}\langle k \rangle^{2s} 
|  u_k(y_0) |^2 \\
& \leq \sum_{k \in \bZd}   e^{2\sigma |k|_1} 
 \int_{-\infty}^0 2 \langle k \rangle^{2s} | \pa_y  u_k(y) | \, | u_k(y)| \, \de y  \\
& \leq 
\int_{-\infty}^0 \sum_{k \in \bZd}    e^{2\sigma |k|_1} \langle k \rangle^{2s-1} | \pa_y u_k(y) |^2 \de y 
+ 
\int_{-\infty}^0  \sum_{k \in \bZd}   
e^{2\sigma |k|_1} \langle k \rangle^{2s+1}  |  u_k(y) |^2 \de y  \\
& = \| u \|_{L^2(\bR_{\leq 0}, H^{\sigma,s+\frac12} )}^2 +
 \| \pa_y u \|_{L^2(\bR_{\leq 0}, H^{\sigma,s-\frac12} )}^2  
\end{aligned} 
$$
which proves \eqref{uinfty} for smooth functions with compact support and then by density for all functions. Finally, recalling the definition of the norm $ \| \ \|_{\sigma,s,a}$
in \eqref{migliore}, we deduce  \eqref{est:trace}. \\[1mm]
\noindent
{\bf Proof of Lemma \ref{lem:und.varphi}.}
In view of \eqref{migliore} we have that 
\begin{align*}
\| e^{y|D|}g - g_0 \|_{\sigma, s+\frac12, a}^2
& =
\sum_{j = 0}^{s+ \frac12} \sum_{k \in \bZd\setminus \{0\} }
e^{2\sigma |k|_1 } \,  \langle k \rangle^{2(s+\frac12-j)} \,  |  g_k|^2
\int_{-\infty}^0 |\pa_y^j e^{|k| y} |^2 \, e^{-2 a y} \de y  \\
& = \sum_{j = 0}^{s+\frac12}  \sum_{k \in \bZ^d \setminus \{0\}} 
e^{2\sigma |k|_1 } \, \langle k \rangle^{2s+1} \langle k \rangle^{-2j}   \,   |  g_k|^2
|k|^{2j} \int_{-\infty}^0 e^{2(|k|-a) y}  \de y \\
& = ({s+\frac12})  \sum_{k \neq 0} 
\frac{\langle k \rangle^{2s+1}}{2(|k| - a)} \, e^{2\sigma |k|_1 } \,  | g_k|^2 \leq C_{a,s} 
\|g\|_{H^{\sigma, s}}^2 
\end{align*}
proving the lemma.\\[1mm]
\noindent
{\bf Proof of Proposition \ref{lem:algebray}.}
We define
$$
\cH^{\sigma, s, a}_\bR := \left\lbrace  u \;:\; \bT^d \times \bR \to \bC \;:\; \ \ 
 \|  u \|_{\sigma, s, a,\bR} < \infty
\right\rbrace
$$
endowed with the norm 
\begin{align}
\| u \|_{\sigma, s, a,\bR}^2 & := 
\sum_{j = 0}^s 
\| \pa_y^j u \|_{L^{2,a}( \bR,H^{\sigma,s-j})}^2 \notag  \\ 
& = 
\sum_{j = 0}^s \int_{\bR} \sum_{k \in \bZd}
 \, e^{2 \sigma |k|_1 } \, \langle k \rangle^{2(s-j)} \, 
| \pa_y^j  u_k  (y)|^2 e^{2 a |y|} \de y  \notag \\  
& = 
\sum_{j = 0}^s \sum_{k \in \bZd}
e^{2 \sigma |k|_1 } \, \langle k \rangle^{2(s-j)} \,  \| \pa_y^j  u_k \|_{L^{2,a}(\bR)}^2   
\label{migliore4R}
\end{align}
where,  given a Hilbert space $ X $, we have used the notation
$$
\| u \|_{L^{2,a}( \bR,X)}^2  := 
 \int_{\bR} \| u(y) \|_X^2 e^{2 a |y|} \de y  \, .
$$
By adapting the method in \cite{LM} of ``extension by reflection" we have the following lemma.
\begin{lem}\label{lem:ext}
{\bf (Extension operator)} There exists a linear bounded extension operator
$ {\cal E}_s: \cH^{\sigma, s, a} \to \cH^{\sigma, s, a}_\bR $
such that $  {\cal E}_s u = u $ a.e. on $ (- \infty,0) $. Thus 
\begin{equation}\label{stimeriflessione} 
\| u \|_{\sigma, s , a} \leq \| {\cal E}_s u \|_{\sigma, s, a,\bR} \lesssim_s \| u \|_{\sigma, s , a} \, .
\end{equation}
\end{lem}

\begin{proof}
We follow \cite{LM}.
For any $u \in {\cal C}^\infty_c (\bR, H^{\sigma,s}) $  we define
\begin{equation}\label{riflessione} 
(\cal E_s u)(y) := 
\begin{cases}
u(y) , \quad y \leq 0 \, , \\
\alpha_0^{(s)} u(-y) + \dots +\alpha_{s}^{(s)} u(-(s+1)y) \, , \quad y > 0 \, ,
\end{cases}
\end{equation}
where the coefficients $\alpha_j^{(s)}$, $j=0,\dots,s$ are to  be chosen in order to have
$$
\pa_y^j ({\cal E}_s u)(0) =  (\pa_y^j u)(0), \quad \forall j=0,\dots,s \, , 
$$
namely solve the  linear system
$$
\begin{pmatrix} 1 & 1 &\ \dots &  1 \\ 
-1 & -2 & \dots & -s-1 \\ \vdots &\  \vdots & \ddots & \vdots \\
\vdots &  \vdots & \ddots &\ \vdots \\
(-1)^{s} & (-2)^{s} &\ \dots & (-s-1)^{s}
  \end{pmatrix} \begin{pmatrix} \alpha_0^{(s)} \\ \vdots  \\ \vdots \\ \vdots \\ \alpha_s^{(s)} \end{pmatrix} = \begin{pmatrix} 1 \\ \vdots \\ \vdots \\ \vdots \\ 1 \end{pmatrix}.
$$
The above  Vandermonde matrix is invertible 
and thus the coefficients $\alpha_0^{(s)},\dots,\alpha_s^{(s)} $ are uniquely 
well-defined. Then  by  \eqref{riflessione} 
\begin{align*}
\| {\cal E}_s u \|_{\sigma,s,a,\bR}^2 
& \leq 
\|u\|^2_{\sigma,s,a} + 
 C_s \sum_{j=0}^s \sum_{i=0}^s \int_{0}^\infty 
 \|  \pa_y^j u(-(i+1)y) \|_{H^{\sigma,s-j}}^2 e^{2a|y|}\de y  \\
&\leq C_s' \sum_{i,j=0}^s  \int_{-\infty}^0 \| \pa_y^ju(z) \|_{H^{\sigma,s-j}}^2 e^{\frac{2a|z|}{i+1}}\de z  \lesssim_s \|u\|^2_{\sigma,s,a} \, .
\end{align*}
By density the operator $ u \mapsto {\cal E}_s u $ admits a bounded linear extension to 
$ {\cal E}_s: \cH^{\sigma, s, a} \to \cH^{\sigma, s, a}_\bR $ and the lemma follows.   
\end{proof}

In order to analyze the space $\cH^{\sigma,s,a}_\bR  $ we can  
use the Fourier transform in the variable $ y $.  
Given  a function $ y \mapsto u (y)  $ in 
$ L^2 (\bR, X)$, where $ X $ is a Hilbert space, we denote its Fourier transform
$$ 
\widehat u(\xi) := ({\cal F}u) (\xi) := \int_\bR \widehat{u}(y) e^{-\im \xi \, y} \de y
$$
and, by the inverse Fourier transform formula,  
$$
u(y) =  \int_\bR \widehat{u}(\xi) e^{\im \xi \, y} \rdj \xi \, \, , \quad 
\rdj :=  \frac{1}{2\pi}  \de \, . 
$$
When $ X = H^{\sigma,s}(\bT^d) $ we may also Fourier expand in $ x $ writing  
\begin{equation}\label{Fourierinverse}
u(y,x) = \sum_{k\in\bZd} u_k (y) e^{\im k\cdot x} =  \sum_{k\in\bZd} \int_\bR 
\widehat{u}_k(\xi) e^{\im (\xi y+k\cdot x)}\rdj \xi \, . 
\end{equation}
In the sequel we shall also denote $\widehat{u}_k(\xi) = \widehat{u}(k,\xi) $. 

The following lemma characterizes the space $\cH^{\sigma,s,a}_\bR$.
\begin{lem} {\bf (Characterization of $ \cH^{\sigma,s,a}_\bR$)}\label{lem:car2}
The space
$ \cH^{\sigma,s,a}_\bR $, $ a > 0 $,  coincides with the  functions 
$ y \mapsto u(y) \in L^2 (\bR, H^{\sigma,s})$ whose Fourier transform $ \xi \mapsto  \widehat u (\xi ) $ 
admits an extension $ \widehat u( \zeta ) $ in the complex strip 
$ 
\{  \zeta = \xi + \im \eta \, \colon  \, \, \xi \in \bR \, , \; \eta \in \bR,\, \ 
|\eta| \leq a \} $, analytic in $ |\eta | < a $, 
whose traces   at the boundaries  $ \widehat u( \cdot  + \varsigma \im a ) $,  
belong to  $ L^2(\bR, H^{\sigma,s})$ and 
$
\langle\xi\rangle^{j} \widehat{u}(\xi+\varsigma \im a) \in 
{L^2(\bR,H^{\sigma,s-j})} $, for any $ j=0,\dots,s $, $ \varsigma=\pm 1$,  
with equivalence of the norms
\begin{align}
\label{normequiv1}
 \|u\|_{\sigma,s,a,\bR}^2 
 & \simeq_{s,a}
\max_{\substack{j=0,\dots,s\\ \varsigma=\pm 1}}  \big\| 
\langle\xi\rangle^{j} \widehat{u}(\xi+\varsigma \im a)  \big\|_{L^2(\bR,H^{\sigma,s-j})}^2  \\
 & \simeq_{s,a} \max_{\substack{j=0,\dots,s\\ \varsigma=\pm 1}} 
 \sum_{k\in\bZd} \int_\bR \langle\xi\rangle^{2j} \langle k \rangle^{2(s-j)} e^{2\sigma |k|_1}   |\widehat{u}_k (\xi+\varsigma \im a)|^2 \de \xi   \notag \\
 & \simeq_{s,a} \max_{\varsigma=\pm 1}
\sum_{k\in\bZd} 
\int_\bR  \langle k,\xi \rangle^{2s} 
e^{2\sigma |k|_1}  |\widehat{u}_k(\xi+\varsigma \im a)|^2  
\de \xi    \label{normequiv2} 
 \end{align}
 where $\langle k,\xi\rangle := \sqrt{1 +|k|^2 +  |\xi|^{2}} $.
\end{lem}
\begin{proof}
For any $|\eta| \leq a $, integrating by parts, 
\begin{equation*}
\begin{aligned} 
(-\im \xi)^j \widehat u(\xi\pm\im\eta) 
& = 
 \int_\bR u(y) e^{\mp \eta y} (-\im\xi)^j e^{-\im \xi y} \de y 
 = \int_\bR u(y) e^{\mp \eta y} \pa_y^j (e^{-\im \xi y}) \de y \\
 & = (-1)^j \int_\bR \pa_y^j \big( u(y) e^{\mp \eta y} \big)  e^{-\im \xi y} \de y 
 =\int_\bR \Big(  \sum_{p=0}^j {j\choose p} (\mp\eta)^{j-p} e^{\mp \eta y}   \pa_y^p u(y) \Big)  e^{-\im \xi y}\de y \, . 
 \end{aligned}
 \end{equation*}
By  Plancherel theorem 
we have 
\begin{equation}\label{xiupp}
\| \xi^j  \widehat u(\xi\pm\im \eta) \|_{L^2_\xi(\bR,H^{\sigma,s-j})}^2 
 \lesssim \sum_{p=0}^j \int_{\bR} 
  \| \pa_y^p u (y) \|_{H^{\sigma,s-j}}^2 e^{2|y|a} \de y 
  \lesssim
  \| u \|_{\sigma,s,a,\bR}^2 \, . 
\end{equation}
Conversely, for any $ 0 \leq j \leq s $, we have 
\begin{align}
 \| e^{a|y|}\pa_y^j  u \|_{L^{2}(\bR,H^{\sigma,s-j})}^2 &
 = \int_0^\infty e^{2a y} \| \pa_y^j u(y) \|_{H^{\sigma,s-j}}^2\de y  +  \int_{-\infty}^0 e^{-2a y} 
 \| \pa_y^j u(y) \|_{H^{\sigma,s-j}}^2\de y \notag \\
&\leq   \| e^{a y} \pa_y^j u(y)\|_{L^2(\bR,H^{\sigma,s-j})}^2  + 
\| e^{-a y} \pa_y^j u(y)\|_{L^2(\bR,H^{\sigma,s-j})}^2\, .  \label{daqua}
\end{align}
Now
\begin{align*}
 {\cal F} \big(e^{\pm a y} \pa_y^j u(y)\big)(\xi) &= \int_\bR  
 e^{-\im \xi y} e^{\pm a y} \pa_y^j u(y)
  \de y = (-1)^j \sum_{p=0}^j {j \choose p} (\pm a)^{j-p} (-\im \xi)^{p} \int_\bR e^{-\im \xi y} e^{\pm ay} u(y)\de y 
 \\ &= \sum_{p=0}^j {j \choose p} (\mp a)^{j-p} (\im \xi)^{p} \widehat{u}(\xi\mp \im a) \, ,
\end{align*}
and thus, by Plancharel theorem,
\begin{align}\label{appsop}
\| e^{\pm a y} \pa_y^j u(y) \|_{L^2(\bR,H^{\sigma,s-j})}^2
&  \simeq
\| {\cal F} \big(e^{\pm a y} \pa_y^j u(y)\big)(\xi)\|_{L^2(\bR,H^{\sigma,s-j})}^2  \notag 
\\
& \lesssim_{a,j} \| \langle\xi\rangle^{j} \widehat{u}(\xi\mp \im a)\|_{L^2(\bR,H^{\sigma,s-j})}^2 \, .
\end{align}
We deduce by \eqref{daqua} and \eqref{appsop} that, for any $ j = 0, \ldots, s $,  
\begin{equation}\label{appsop2}
 \| e^{a|y|}\pa_y^j  u \|_{L^{2}(\bR,H^{\sigma,s-j})}^2  \lesssim_{a,s}
 \max_{\varsigma = \pm}\| \langle\xi\rangle^{j} 
 \widehat{u}(\xi + \varsigma \im a)\|_{L^2(\bR,H^{\sigma,s-j})}^2 \, . 
\end{equation}
The estimates \eqref{xiupp} and \eqref{appsop2}  prove \eqref{normequiv1}. 

The last equivalence in \eqref{normequiv2}  follows by  Young's inequality
$ \langle\xi\rangle^{2j}\langle k\rangle^{2(s-j)} \leq \langle \xi\rangle^{2s} + \langle k\rangle^{2s} $.
\end{proof}

We now prove the tame estimate for the product of two functions in $\cH^{\sigma,s,a}_\bR $.

\begin{lem}[\bf Tame]\label{lem:tame} Let $\sigma,a \geq 0 $,  $ s, s_0 \in \bN $ 
such that 
 $ s \geq s_0 > \frac{d+1}{2}$.
Then
\begin{equation}\label{tamecilindro}
\| u v \|_{\sigma, s, a,\bR} \leq
C_s \,  \| u  \|_{\sigma, s, 0,\bR} \| v \|_{\sigma, s_0, a,\bR}
+
C_s \, 
\| u  \|_{\sigma,s_0,0,\bR} \| v \|_{\sigma,s,a,\bR}\, .
\end{equation}
\end{lem}

\begin{proof}
The product of the functions (cfr. \eqref{Fourierinverse})   
\begin{equation*}
u(x,y) = \sum_{k\in\bZd} \int_\bR \widehat{u}(k, \xi) 
e^{\im (\xi y+k\cdot x)}\rdj \xi \, , \quad  v(x,y) = \sum_{k\in\bZd} \int_\bR \widehat{v}(k,\xi) e^{\im (\xi y+ k \cdot x)}\rdj \xi \, , 
\end{equation*}
is 
$ u v = 
\sum_{m \in \bZd} \int_\bR \widehat{uv}(m,\eta) e^{\im \eta y} \, e^{\im m x } \rdj \eta $
with
\begin{equation}\label{Fourierprod}
\widehat{uv}(m,\eta)= \sum_{k \in \bZd}
\int_\bR
\widehat u(k, \xi) \, 
\widehat v(m-k, \eta - \xi) \rdj \xi \, . 
\end{equation}
By \eqref{normequiv2}, \eqref{Fourierprod} 
we have that 
\begin{align}\label{I1}
\|uv\|_{\sigma,s,a,\bR}^2&\simeq 
\max_{\varsigma\in\{\pm 1\}} \sum_{m\in\bZd} \int_\bR  
| \widehat{uv}(m,\eta+\varsigma \im a)|^2 e^{2\sigma |m|_1} 
\langle m,\eta \rangle^{2s} \rdj\eta  \\
 & \leq \max_{\varsigma\in\{\pm 1\}}  \sum_{m\in\bZd} \int_\bR  \Big( \sum_{k \in \bZd}
\int_\bR
|\widehat u(k, \xi)| \, 
|\widehat v(m-k, \eta - \xi+\varsigma \im a)| 
 e^{\sigma |m|_1} \langle m,\eta \rangle^{s}
 \rdj \xi\Big)^2 \! \rdj\eta. \notag
\end{align}
We split the frequency  space into 
\begin{align*} A&:= \big\{(m,\eta,k,\xi)\in\bZ^d\times\bR\times \bZ^d\times \bR \;:\; \langle m,\eta \rangle \leq 2 \langle k,\xi  \rangle \big\},\\ B&:= \big\{(m,\eta,k,\xi)\in\bZ^d\times\bR\times \bZ^d\times \bR\;:\; \langle m,\eta \rangle > 2 \langle k,\xi  \rangle \big\}   \end{align*}
and, since $e^{\sigma|m|_1}\leq e^{\sigma|k|_1}e^{\sigma|m-k|_1} $, 
we estimate \eqref{I1} as
\begin{equation}\label{smalls-1}
\|uv\|_{\sigma,s,a,\bR}^2\leq I_1+I_2 
\end{equation}
where
{\footnotesize{\begin{align*}
I_1 &:= \sum_{m\in\bZd} \int_\bR  \Big(\!\!\sum_{k \in \bZd}\!
\int_{\bR_A} \!\!\!\!\! e^{\sigma |k|_1}
|\widehat u(k, \xi)| \langle k,\xi \rangle^{s} \, e^{\sigma |m-k|_1}
|\widehat v(m-k, \eta - \xi+\varsigma \im a)| \langle m-k,\eta-\xi \rangle^{s_0} \frac{\langle m,\eta  \rangle^s}{\langle k,\xi\rangle^{s}\langle m-k,\eta-\xi \rangle^{s_0}} \rdj \xi\Big)^2 \!\!\! \rdj\eta, \\ 
I_2 &:= \sum_{m\in\bZd} \int_\bR  \Big(\!\!\sum_{k \in \bZd}\!
\int_{\bR_B}\!\!\!\!\! e^{\sigma |k|_1}
|\widehat u(k, \xi)| \langle k,\xi \rangle^{s_0} \, e^{\sigma |m-k|_1}
|\widehat v(m-k, \eta - \xi+\varsigma \im a)| \langle m-k,\eta-\xi \rangle^{s} \frac{\langle m,\eta  \rangle^s}{\langle k,\xi \rangle^{s_0}\langle m-k,\eta-\xi \rangle^{s}} \rdj \xi\Big)^2 \!\!\! \rdj\eta,
\end{align*}}}
where, given $m,\, k\in\bZ^d $ and $\eta\in\bR$, we denoted
$$ 
\bR_A:= \Big\{ \xi\in\bR\;:\; (m,\eta,k,\xi) \in A \Big\}, \qquad \bR_B:= \Big\{ \xi\in\bR \;:\; (m,\eta,k,\xi) \in B \Big\} \, .
$$ 
Using that,  if $(m,\eta,k,\xi)\in A $ then $ \langle k,\xi \rangle >  \frac12 \langle m,\eta \rangle  $, the Cauchy-Schwarz inequality and exchanging the order of integration we get 
\begin{align}
I_1 & \lesssim_{s} \sum_{m\in\bZd}  \int_\bR  \Big( \sum_{k \in \bZd } 
\int_\bR  e^{2\sigma |k|_1}
|\widehat u(k, \xi)|^2 \langle k,\xi \rangle^{2s} \, e^{2\sigma |m-k|_1}
|\widehat v(m-k, \eta - \xi+\varsigma \im a)|^2 \langle m-k,\eta-\xi \rangle^{2s_0}\rdj \xi \Big) \notag \\ & 
\qquad \times \Big( \sum_{k \in \bZd }\int_{\bR _A} \frac{1}{
\langle m-k,\eta-\xi \rangle^{2s_0}} \rdj \xi \Big) \rdj\eta  \notag \\ 
& \lesssim_{s}
\sum_{k\in\bZd} 
\int_\bR  e^{2\sigma |k|_1}
|\widehat u(k, \xi)|^2 \langle k,\xi \rangle^{2s} 
\Big( \int_\bR   \sum_{m\in\bZd}   \, e^{2\sigma |m-k|_1}
|\widehat v(m-k, \eta - \xi+\varsigma \im a)|^2 \langle m-k,\eta-\xi \rangle^{2s_0}
\rdj\eta \Big) \rdj \xi 
\notag  \\
& \stackrel{\eqref{normequiv2}} {\lesssim_{s}}  \|u\|_{\sigma,s,0,\bR}^2 \|v\|_{\sigma,s_0,a,\bR}^2 \, .  
\end{align}
Note that since $ s_0 \in \bN $, $ s_0 > (d + 1 )/2$ 
then $ s_0 \geq \frac{d+1}{2} + \frac12 $. 
If $(m,\eta,k,\xi)\in B$, i.e. $ \langle k,\xi \rangle < \frac12 \langle m,\eta \rangle  $, then
$ \langle m-k,\eta-\xi \rangle > \frac12 \langle m,\eta \rangle  $,   and one deduces 
similarly that 
$$
I_2 \lesssim_{s,s_0} \|u\|_{\sigma,s_0,0,\bR}^2 \|v\|_{\sigma,s,a,\bR}^2 \, , 
$$
proving, in view of \eqref{smalls-1},  the tame estimate \eqref{tamecilindro}.
\end{proof}

We now prove  \eqref{algebrassa}. Given 
$ u  \in \cH^{\sigma, s, 0}$ and  $ v \in \cH^{\sigma, s, a}$, we consider their extensions
$ {\cal E}_s u $ and $ {\cal E}_s v  $ obtained by Lemma \ref{lem:ext}.
Since the product $ {\cal E}_s u  {\cal E}_s v  $  is an extension of $ u v $ 
we have that  
$ \| u v \|_{\sigma,s,a} \leq \| {\cal E}_s u  {\cal E}_s v \|_{\sigma,s,a,\mathbb R} $.
Thus the tame estimate \eqref{tamecilindro} and 
the equivalence of the norms 
 in 
 \eqref{stimeriflessione},
  imply \eqref{algebrassa}.

  The  proof of \eqref{potenze} follows by  induction on $j$. For $j=1$ it is trivial 
  and if it holds for $j$ then
  \begin{align*}\|u^{j+1} \|_{\sigma,s,a} &
  \stackrel{\eqref{algebrassa}}{\leq} C_s \big(\|u\|_{\sigma,s,a}\|u^j\|_{\sigma,s_0,a}+\|u\|_{\sigma,s_0,a}\|u^j\|_{\sigma,s,a} \big) \\
 &  \stackrel{\eqref{potenze}_j}\leq 
 C_s \big(\|u\|_{\sigma,s,a} (2C_{s_0})^{j-1} \|u\|_{\sigma,s_0,a}^j +
 \|u\|_{\sigma,s_0,a} (2C_s \|u\|_{\sigma,s_0,a})^{j-1} \|u\|_{\sigma,s,a} \big)\\
&  \leq (2 C_s)^j \|u\|_{\sigma,s_0,a}^j \|u\|_{\sigma,s,a}   
 \end{align*}
proving \eqref{potenze} at the step $ j + 1 $.

\section{Proof of the elliptic regularity Lemma \ref{lem:ell5}}\label{sec:sol.ell}

We define  the weighted $ L^2 $-space of functions (cfr. \eqref{defuL2a})
\begin{equation}\label{defL2a}
L^{2, a}:= \left\lbrace p \;\colon\; \bR_{\leq 0} \to \bC \ \ \ : \ \  
  \| p \|_{ L^{2,a}}^2 :=  \int_{-\infty}^0  \, |p(y)|^2 e^{-2a y} \de y < \infty  \right\rbrace \, . 
\end{equation}
For any $\lambda \geq 0 $, define also the integral operators
\begin{equation}
 \label{Tk}
( T_\lambda p)(y):= \int_{-\infty}^y e^{\lambda  (z-y)} \, p(z) \de z  \, , 
\quad
( \widetilde T_\lambda  p)(y):= \int_{y}^0 e^{\lambda  (y-z)} \, p(z) \de z \, , 
\quad \forall y \leq 0 \, . 
 \end{equation} 
 The next technical lemma shows that the operators $T_\lambda, 
 \widetilde T_\lambda $ extend to bounded  operators on $L^{2,a}$.
\begin{lem}\label{lem:T}
Let $s \geq 0, a >0$.  For any $\lambda  \geq 0 $ one has
\begin{align}\label{Tn.est1}
& \|T_\lambda  p \|_{ L^{2,a}} \leq \frac{1}{\lambda +a}\, \|p\|_{ L^{2,a}} \, , \\
&  |(T_\lambda p)(y)| \leq \frac{e^{a y}}{\sqrt{2(\lambda  + a)}} \| p \|_{ L^{2,a}} \,,  \quad \forall y \leq 0 \, , \label{Tn.est1bis}
 \\
&
\| \pa_y^j (T_\lambda  p) \|_{ L^{2,a}} \leq 
C_a \sum_{i =0}^{j-1}  \langle \lambda  \rangle^{j-i-1} \| \pa_y^{i} p\|_{ L^{2,a}} \, , 
\quad \forall j \geq 1 \, , \label{Tn.est2}
\end{align}
where $ \langle \lambda \rangle = \max\{1, |\lambda| \}   $. For  any $ \lambda > a  $,  one has
\begin{align}
& \|\widetilde T_\lambda p \|_{ L^{2,a}} \leq \frac{1}{\lambda -  a} \, \|p\|_{ L^{2,a}} \, ,  \label{Tn.est3} \\
&  |(\widetilde T_\lambda p)(y)| \leq \Big( \frac{e^{2a y} - e^{2\lambda  y}}{2(\lambda  - a)}  \label{Tn.est3bis}
\Big)^{\frac12}\| p \|_{ L^{2,a}} \, , \quad \forall y \leq 0 \, ,   \\
&   
\|  \pa_y^j (\widetilde T_\lambda  p) \|_{ L^{2,a}} \leq  C_a \sum_{i =0}^{j-1} 
\langle \lambda  \rangle^{j-i-1} \| \pa_y^{i} p\|_{ L^{2,a}} \, , \quad \forall j \geq 1 \, . 
\label{TB8}
\end{align} 
\end{lem}
\begin{proof}
We consider first the operator $T_\lambda $.
 Recalling \eqref{Tk} we have
\begin{align*}
\| T_\lambda  p \|_{L^{2,a}}^2 = \int_{-\infty}^0 
\left| \int_{-\infty}^y e^{\lambda  (z-y)} p(z) \de z \right|^2 e^{-2 a y}  \, \de y  \leq
 \int_{-\infty}^0 
\left( \int_{-\infty}^y  e^{(\lambda +a)(z-y)} |p(z)| e^{-a z} \de z \right)^2 \de y \, . 
\end{align*}
Since $\int_{-\infty}^y   e^{(\lambda +a)(z-y)}\de z  = \frac{1}{\lambda +a}$, 
the measure 
$ (\lambda +a)  e^{(\lambda +a)( z-y)} \de z$ is  normalized on the domain $ (-\infty,y) $, and Jensen inequality and exchanging the order of integration implies 
\begin{align*}
\| T_\lambda  p \|_{L^{2,a}}^2 
&\leq 
\frac{1}{(\lambda +a)^2} \int_{-\infty}^0  
 \int_{-\infty}^y  (\lambda +a) e^{(\lambda +a)( z-y)} \,  |p(z)|^2 e^{-2a z}  \de z  \de y
\\
&  = \frac{1}{\lambda +a}\int_{-\infty}^0 e^{-2az} \,  |p(z)|^2 e^{(\lambda +a)z} \left( 
\int_z^0 e^{-(\lambda + a   ) y } \de y\right)  \de z  \\
 & = \frac{1}{(\lambda +a)^2} \int_{-\infty}^0 e^{-2az}  |p(z)|^2 \left( 1- e^{(\lambda +a)z} \right)  \de z \leq  \frac{1}{(\lambda +a)^2} \| p \|_{L^{2,a}}^2
\end{align*}
as $ 1-e^{(\lambda +a)z} \leq 1$ for any  $ z \leq 0 $. 
This proves \eqref{Tn.est1}. Estimate \eqref{Tn.est1bis} descends, recalling \eqref{Tk},
\eqref{defL2a}, and applying Cauchy-Schwarz inequality,
$$
\begin{aligned}
| (T_\lambda  p )(y) | & \leq e^{-\lambda y} \int_{-\infty}^y e^{\lambda  z} |p(z)| \de z \\
&  \leq 
e^{-\lambda y}\left( \int_{-\infty}^y e^{2(\lambda +a)z} \de z \right)^{\frac12} \, 
\Big( \int_{-\infty}^0  \, |p(z)|^2 e^{-2a z} \de z \Big)^{\frac12}
= \frac{e^{ay}}{\sqrt{2(\lambda +a)}} \| p \|_{L^{2,a}} \, .
\end{aligned}
$$
In order to prove the estimate \eqref{Tn.est2}  for  $\pa_y^j (T_\lambda  p)$, we first note
the following inductive formula
$$
\pa_y^j \,  (T_\lambda  p) =  
\sum_{i=0}^{j-1} (-\lambda )^{j-i-1} \,  \pa_y^{i} p   + (-\lambda )^j T_\lambda  p \, , \quad 
\forall j \geq 1 \, . 
$$
Then \eqref{Tn.est2}  follows estimating $ \| T_\lambda  p \|_{L^{2,a}} $ by  \eqref{Tn.est1}.

Now we consider the operator $\widetilde T_\lambda $ in \eqref{Tk}.  Since $\int_y^0   e^{(\lambda -a)(y-z)}\de z  = \frac{1-e^{(\lambda -a)y}}{\lambda -a}$, then, for any $ \lambda  > a $, 
the measure 
$ \frac{\lambda -a}{1-e^{(\lambda -a)y}}  e^{(\lambda -a)(y-z)} \de z$ is  normalized on the domain $ (y,0) $, and, 
by Jensen inequality and exchanging the order of integration, 
\begin{align*}
\| \widetilde T_\lambda p \|_{L^{2,a}}^2 
&\leq  \int_{-\infty}^0 \frac{1-e^{(\lambda-a)y}}{\lambda-a}  \int_y^0 |p(z)|^2e^{-2az} e^{(\lambda-a)(y-z)} \de z\de y \\
&= \frac{1}{\lambda - a} \int_{-\infty}^0 |p(z)|^2 e^{-2az}e^{-(\lambda-a)z} \int_{-\infty}^z \big(e^{(\lambda-a)y}-e^{2(\lambda-a)y} \big) \de y\de z \\
&= \frac{1}{(\lambda-a)^2}\int_{-\infty}^0 |p(z)|^2 e^{-2az}\big(1-\frac12 e^{(\lambda-a)z} \big)\de y \leq \frac{1}{(\lambda-a)^2} \|p\|_{L^{2,a}}^2 
\end{align*} 
as $1-\frac12 e^{(\lambda-a)z}\leq 1 $  for any $z\leq 0$. This proves \eqref{Tn.est3}. The estimate \eqref{Tn.est3bis} follows similarly to \eqref{Tn.est1bis}. Finally \eqref{TB8} 
descends from the identity 
$$ 
\pa_y^j \,  (\widetilde T_\lambda p) =  
- \sum_{i=0}^{j-1} \lambda^{j-i-1} \,  \pa_y^{i} p   + \lambda^j {\widetilde T}_\lambda p 
\, , \quad \forall j \geq 1 \, ,
$$ 
together with the estimate for ${\widetilde T}_\lambda p $ in \eqref{Tn.est3}.
\end{proof}

\begin{proof}[Proof of Lemma \ref{lem:ell5}]
Writing $ u(x,y) = \sum_{k \in \bZ^d} u_k (y) e^{\im k \cdot x} $, 
we expand \eqref{ell5} in Fourier in the $x$ variables, obtaining for any $k \in \bZ^d$ the second order system for $ u_k(y)$, 
\begin{equation}\label{ell6}
\begin{cases}
-|k|^2   u_k(y) + \pa_y^2  u_k (y) =  g_k(y)  \\
 u_k(0) =0 ,  \quad 
\pa_y   u_k(y) \to 0 \mbox{ as } y \to - \infty \, .  
\end{cases}
\end{equation}
\underline{Case $k = 0$:} 
The solution of \eqref{ell6} is, for $ k = 0 $, 
\begin{equation}
\label{u0}
u_0(y) = \int_{-\infty}^y \int_{-\infty}^{y'} g_0(z) \de z \de y'  - \int_{-\infty}^0 \int_{-\infty}^{y'} g_0(z) \de z\de y'
\stackrel{\eqref{Tk}} = \underbrace{(T_0^2 g_0)(y)}_{=\Pi u_0} \underbrace{- (T_0^2 g_0)(0)}_{=u_0 - \Pi u_0} \, . 
\end{equation}
First note  that, since $ g \in \cH^{\sigma,s,a} $, then, by  \eqref{migliore4},
\begin{equation}\label{payjg0}  
\| \pa_y^j g_0 \|_{L^{2,a}} \leq \| g \|_{\sigma,s,a} \, , \quad \forall  j = 0, \ldots, s \, . 
\end{equation}  
By 
\eqref{Tn.est1}, \eqref{Tn.est1bis}, 
the function $ \Pi u_0 = T_0^2 g_0 $ and the constant  
$ u_0 - \Pi u_0 = -   (T_0^2 g_0)(0)  $ satisfy 
\begin{equation}
\label{stimau0} 
 \begin{aligned}
& \| \Pi u_0  \|_{L^{2,a}}  
  \leq  a^{-2} \|g_0\|_{L^{2,a}} 
  \leq a^{-2} \| g \|_{\sigma,s,a} \, , \ 
 | u_0 - \Pi u_0  | 
\leq \frac{1}{\sqrt{2} a^{3/2} } \, \| g_0\|_{L^{2,a}} \leq \frac{1}{\sqrt{2} \, a^{3/2} } \| g \|_{\sigma,s,a}  \ . 
 \end{aligned}
 \end{equation}
 Thus $u_0 \in \bC \oplus L^{2,a}$.  
  Moreover    $\pa_y u_0 = T_0 g_0 $ and  \eqref{Tn.est1bis}, \eqref{Tn.est1} imply that
 \begin{equation}\label{1payu}
 |(\pa_y u_0)(y)| \leq \frac{e^{a y}}{\sqrt{2a}} \| g_0 \|_{ L^{2,a}} \, , \quad 
\|\pa_y u_0\|_{L^{2,a}} \leq a^{-1} \|g_0\|_{L^{2,a}} 
 \leq a^{-1} \| g \|_{\sigma,s,a}   \, . 
\end{equation} 
In addition,  since $ \pa_y^2  u_0 = g_0 $, 
 we get   $\pa_y^j u_0 = \pa_y^{j-2} g_0$, for any $ j \geq 2$, and then, by \eqref{payjg0}, 
\begin{equation}\label{j+20}
\| \pa_y^j u_0 \|_{L^{2,a}}  =  \|\pa_y^{j-2} g_0\|_{L^{2,a}}   \leq \| g \|_{\sigma, s, a} \, , 
\quad \forall 2 \leq j \leq s + 2 \, . 
\end{equation}
The bounds
\eqref{stimau0},
\eqref{1payu}, 	\eqref{j+20},  and recalling \eqref{migliore4}, imply that  
\begin{equation}
\label{u.est0}
| u_0 - \Pi u_0|^2 + \| \Pi u_0 \|_{L^{2,a}}^2 + 
\sum_{j=1}^{s+2} 
\| \pa_y^j u_0 \|_{L^{2,a}}^2
 \leq C_{a} \| g \|_{\sigma, s, a}^2 \, . 
\end{equation}
\noindent\underline{Case $k \neq 0$:} 
The solution of the linear equation \eqref{ell6} 
is (by the variation of constants method)
\begin{align}
\notag
u_k(y) & =
- \frac{1}{2|k|} \int\limits_{-\infty}^y e^{|k|(z-y)} \, g_k(z) \, \de z 
- \frac{1}{2|k|} \int\limits_y^0 e^{|k|(y-z)} g_k(z) \, \de z 
+\frac{e^{|k| y} }{2|k|} \int\limits_{-\infty}^0 g_k(z) e^{|k|z} \de z  
 \\ \label{uk_formula}
& \stackrel{\eqref{Tk}} = - \frac{1}{2|k|} (T_{|k|}g_k)(y) 
- \frac{1}{2|k|} (\widetilde T_{|k|}g_k) (y) 
+ \frac{ e^{|k| y}}{2|k|} (T_{|k|}g_k)(0)  \, .
\end{align}
By \eqref{migliore4}, each 
$ g_k \in L^{2,a} $ and $ \| g_k \|_{L^{2,a}} \leq \| g \|_{\sigma,s,a} $. 
Thus by  $ \| e^{|k|y}\|_{L^{2,a}} = 1/ \sqrt{2(|k|-a)} $, Lemma \ref{lem:T},  
and recalling that $ a \in (0,1) $,
 we bound \eqref{uk_formula} for any $ |k| \geq 1$,  as 
\begin{equation}\label{stimaukf}
\| u_k\|_{L^{2,a}} \leq 
\frac{1}{2|k|} \|T_{|k|} g_k\|_{L^{2,a}}   + 
\frac{ 1 }{2|k|}\|\widetilde T_{|k|} g_k\|_{L^{2,a}}
+ \frac{ |(T_{|k|}g_k)(0) |}{2|k| \sqrt{2(|k|-a)}} |    \lesssim_a
 \frac{1}{|k|^2} \|g_k\|_{L^{2,a}}   \, . 
\end{equation}
Thus each $ u_k \in L^{2,a} $, $ k \neq 0 $. 
Note also that  
$ \pa_y u_k (y) =  \frac{1}{2} (T_{|k|}g_k)(y) 
-
\frac{1}{2} (\widetilde T_{|k|}g_k) (y) 
+
\frac{ e^{|k| y}}{2} (T_{|k|}g_k)(0) $
satisfies, by \eqref{Tn.est1bis} and \eqref{Tn.est3},
$$
\begin{aligned}
|\pa_y u_k (y)| & \leq  \frac{1}{2} |(T_{|k|}g_k)(y)| 
+
\frac{1}{2} |(\widetilde T_{|k|}g_k) (y)| 
+
\frac{ e^{|k| y}}{2} |(T_{|k|}g_k)(0)| \\
& \leq
 \frac12 \frac{e^{a y}}{\sqrt{2(|k|+a)}} \| g_k \|_{ L^{2,a}}+
\frac12 \Big( \frac{e^{2a y} - e^{2|k| y}}{2(|k| - a)} 
\Big)^{\frac12}\| g_k \|_{ L^{2,a}} + 
\frac{ e^{|k| y}}{2} \frac{1}{\sqrt{2(|k|+a)}} \| g_k \|_{ L^{2,a}}
\end{aligned}
$$
thus tends to $  0 $ as $  y \to - \infty $.

By \eqref{stimaukf} and recalling  \eqref{migliore4} we deduce that 
\begin{equation}\label{parzap}
\sum_{k\neq 0 } e^{2\sigma |k|_1}\, \langle  k \rangle^{2(s+2)} \,  
\| u_k \|_{L^{2,a}}^2  \lesssim_a  \|g \|_{\sigma, s,a}^2 
\end{equation}
and we conclude that  
$ u = u_0 (y) + \sum_{k \neq 0} u_k (y) e^{\im k \cdot x} $ 
is in $ \bC \oplus L^{2,a}(\bR_{\leq 0}, H^{\sigma,s+2}) $
with
\begin{equation} \label{u.estk1}
\| \Pi u \|_{ L^{2,a}(\bR_{\leq 0}, H^{\sigma,s+2})}^2 
= \| \Pi u_0 
 \|_{L^{2,a}}^2 +
\sum_{k 
\neq 0 }  e^{2\sigma |k|_1}\, \langle  k \rangle^{2(s+2)} \,   
\|  u_k \|_{L^{2,a}}^2  
\stackrel{\eqref{stimau0},\eqref{parzap}} \leq C_a \|g \|_{\sigma, s,a}^2 \, . 
\end{equation}
Now we estimate the derivatives $ \pa_y^j u $, $j \geq 1$.
Differentiating \eqref{uk_formula}  we get, for any $ j \geq 1 $,  
$$
\pa_y^j u_k(y) = - \frac{1}{2|k|} \pa_y^j (T_{|k|}g_k)(y) -
\frac{1}{2|k|}\pa_y^j  (\widetilde T_{|k|} g_k) 
+
 \, \frac12 |k|^{j-1}\, e^{|k| y} (T_{|k|}g_k)(0)  
$$
and, using  $ \| e^{|k|y}\|_{L^{2,a}} = 1/ \sqrt{2(|k|-a)} $, Lemma \ref{lem:T}, $ a \in (0,1) $,  we get, for any $ |k| \geq 1 $, 
\begin{align}
\notag
\| \pa_y^j u_k\|_{L^{2,a}} & \leq 
\frac{1}{2|k|}\|\pa_y^j T_{|k|} g_k \|_{L^{2,a}}  + 
\frac{1}{2|k|}\|\pa_y^j \widetilde T_{|k|} g_k\|_{L^{2,a}}
+ C_a |k |^{j-\frac32}    \, |(T_{|k|}g_k)(0) | \\
& \stackrel{\eqref{Tn.est2}, \eqref{TB8}, \eqref{Tn.est1bis}}  
{\lesssim_a} \sum_{i =0}^{j-1}  \langle k \rangle^{j-i-2} \,  \| \pa_y^{i} g_k \|_{L^{2,a}} \, . \label{pauj} 
\end{align}
 By \eqref{pauj}  we conclude that, for any $ 1 \leq j \leq s + 1 $, 
\begin{equation}\label{u.est2}
\begin{aligned}
\sum_{k \neq 0 } e^{2\sigma |k|_1}  \langle  k \rangle^{2(s+2 -j)} 
\| \pa_y^j u_k\|_{L^{2,a}}^2 
&  \lesssim_a  \sum_{i= 0}^{j-1}
  \sum_{k \neq 0}  e^{2\sigma |k|_1}  \, 
 \langle  k \rangle^{2(s+2 -j)}  \langle  k \rangle^{2(j-i-2)} \| \pa_y^i g_k \|_{L^{2,a}}^2 \\
 &  \lesssim_a   \sum_{i= 0}^{j-1}
  \sum_{k \neq 0} 
   e^{2\sigma |k|_1}  \, 
 \langle  k \rangle^{2(s-i)}  \| \pa_y^i g_k \|_{L^{2,a}}^2 
   \stackrel{\eqref{migliore4}}\leq C_{a} \| g \|_{\sigma, s, a }^2 \, .
  \end{aligned}
\end{equation}
We finally estimate the last derivative $\partial_y^{s+2} u_k $. 
Differentiating \eqref{ell6} with respect to $ \partial_y^s $, we get  
$$
\partial_y^{s+2} u_k (y) = \partial_y^s g_k (y) + |k|^2 \partial_y^s u_k (y)
$$ 
and then 
\begin{align}
\sum_{k \in \bZ^d} e^{2 \sigma |k|_1} \| \partial_y^{s+2} u_k \|_{L^{2,a}}^2 
& \lesssim
\sum_{k \in \bZ^d} e^{2 \sigma |k|_1}  
\Big( \|  \partial_y^s g_k \|_{L^{2,a}}^2 + |k|^4 \| \partial_y^s u_k \|_{L^{2,a}}^2\Big)  \notag \\
& \stackrel{\eqref{migliore4},\eqref{pauj}} {\lesssim_a} 
\| g \|_{\sigma,s,a}^2 + \sum_{k \in \bZ^d} e^{2 \sigma |k|_1}   \langle k \rangle^4 
\sum_{i =0}^{s-1}  \langle k \rangle^{2(s-i-2)} \,  \| \pa_y^{i} g_k \|_{L^{2,a}}^2 \notag \\
&   \lesssim_a 
\| g \|_{\sigma,s,a}^2 + 
\sum_{i =0}^{s-1}  \sum_{k \in \bZ^d} e^{2 \sigma |k|_1}   \langle k \rangle^{2(s-i)} \,  \| \pa_y^{i} g_k \|_{L^{2,a}}^2 
\stackrel{\eqref{migliore4}} \leq C_{a} \| g \|_{\sigma,s,a}^2 \, . \label{lastder2}
\end{align}
Recalling \eqref{migliore}, summing the estimates \eqref{u.est0}, \eqref{u.estk1}, 
\eqref{u.est2} and \eqref{lastder2}, 
we deduce that  $u \in \bC \oplus \cH^{\sigma, s+2, a} $ 
and $ \| u \|_{\sigma, s+2,a} = | u - \Pi u | + \| \Pi u \|_{\sigma,s+2,a} \leq  C_{s,a} \| g \|_{\sigma,s,a} $. 
Lemma \ref{lem:ell5} is proved. 
\end{proof}

\begin{footnotesize}

 \end{footnotesize}

 International School for Advanced Studies (SISSA), Via Bonomea 265, 34136, Trieste, Italy. 
 \textit{Emails: } \texttt{berti@sissa.it},  \texttt{alberto.maspero@sissa.it}, \texttt{paolo.ventura@sissa.it}

\end{document}